\documentclass[A4paper, 12pt]{article} % use larger type; default would be 10ptt
\usepackage[T1]{fontenc}
\usepackage[utf8]{inputenc} % set input encoding (not needed with XeLaTeX)
\usepackage[all,cmtip]{xy}
\usepackage[new]{old-arrows}
\usepackage[margin=1in, top=1.2in, bottom=1.4in]{geometry}
\usepackage{graphicx} % support the \includegraphics command and options

\usepackage{booktabs} % for much better looking tables
\usepackage{array} % for better arrays (eg matrices) in maths
\usepackage{paralist} % very flexible & customisable lists (eg. enumerate/itemize, etc.)
\usepackage{verbatim} % adds environment for commenting out blocks of text & for better verbatim
\usepackage{subfig} % make it possible to include more than one captioned figure/table in a single float
\usepackage{mathtools}
\usepackage{amssymb}
\usepackage{changepage}
\usepackage{amsmath}
\usepackage{amsthm}
\usepackage{textcomp}
\usepackage{mathdots}
\usepackage{mathrsfs}
\usepackage{enumitem}
\usepackage[colorlinks, allcolors=blue]{hyperref}
\usepackage{appendix}
\usepackage{stmaryrd}
\usepackage{hyperref}
\usepackage[bottom]{footmisc}
\hypersetup{colorlinks,linkcolor={blue},citecolor={blue},urlcolor={black}} 

\usepackage{subfiles}

\usepackage[tightpage,psfixbb]{preview}
\usepackage{titlesec}

\titleformat{\section}
  {\Large} % Large and bold
  {\thesection} % Section numbering
  {0.5em} % Space between number and title
  {} % No additional formatting

\titleformat{\subsection}[runin] % 'runin' makes it inline
  {\normalfont\bfseries} % Same size as body text and bold
  {\thesubsection} % Numbering without a period
  {0.5em} % Space between number and title
  {} % No automatic formatting

\titlespacing{\subsection} % Adjust spacing
  {0pt} % No left indentation
  {10pt} % Space before the title
  {0.5em} % Space after the title before text starts

\titleformat{\subsubsection}[runin]
  {\normalfont\itshape} % Italicized
  {\thesubsubsection} % Number without a period
  {0.5em} % Space between number and title
  {}

\titlespacing{\subsubsection}
  {0pt} % No left indentation
  {10pt} % Space before the title
  {0.5em} % Space after the title before text starts

\usepackage[nottoc,notlof,notlot]{tocbibind} % Put the bibliography in the ToC
\usepackage[titles,subfigure]{tocloft} % Alter the style of the Table of Contents
\theoremstyle{sltheoremstyle}
\newtheorem{theorem}{Theorem}[section]
\newtheorem{lemma}[theorem]{Lemma}

\newtheorem{corollary}[theorem]{Corollary}
\newtheorem{question}[theorem]{Question}
\newtheorem{proposition}[theorem]{Proposition}
\newtheorem{conjecture}[theorem]{Conjecture}

\newtheorem*{corollary-non}{Corollary}
\newtheorem*{lemma-non}{Lemma}
\newtheorem*{theorem-non}{Theorem}
\newtheorem*{proposition-non}{Proposition}
\newtheorem*{condition-non}{Condition}
\newtheorem*{conditions-non}{Conditions}

\theoremstyle{definition}

\newtheorem{remark}[theorem]{Remark}

\newtheorem{definition}[theorem]{Definition}

\newtheorem{example}[theorem]{Example}
\newtheorem{construction}[theorem]{Construction}

\usepackage[x11names]{xcolor}
\usepackage{tabularx}
\usepackage{geometry}
\usepackage{setspace}
\usepackage{lipsum}
\usepackage{authblk}

\usepackage[backend=bibtex,style=alphabetic]{biblatex}

\usepackage{tikz} 
\usepackage{tikz-cd}

\newcommand{\set}[1]{\left\{ #1 \right\}}
\newcommand{\va}[1]{\left| #1 \right|}

\newcommand{\Sch}{\mathsf{Sch}}

\newcommand{\Stab}{\textnormal{Stab}}

\newcommand{\CC}{\mathbb{C}}

\newcommand{\FF}{\mathbb{F}}

\newcommand{\QQ}{\mathbb{Q}}

\newcommand{\RR}{\mathbb{R}}

\newcommand{\ZZ}{\mathbb{Z}}

\newcommand{\NN}{\mathbb{N}}

\newcommand{\id}{\textnormal{id}}

\newcommand{\Gal}{\textnormal{Gal}}

\newcommand{\Ima}{\textnormal{Im}}

\newcommand{\Aut}{\textnormal{Aut}}

\newcommand{\Spec}{\textnormal{Spec}}

\newcommand{\ca}[1]{{\mathcal{#1}}}
\newcommand{\bb}[1]{{\mathbb{#1}}}

\newcommand{\mr}[1]{{\mathscr{#1}}}
\newcommand{\mf}[1]{{\mathfrak{#1}}}
\newcommand{\tn}[1]{{\textnormal{#1}}}
\let\rm\relax 
\newcommand{\rm}[1]{{\mathrm{#1}}}

\newcommand{\ul}[1]{\underline{#1}}

\DeclareSymbolFont{bbm}{U}{bbm}{m}{n}
\DeclareSymbolFontAlphabet{\mathbbm}{bbm}
\newcommand{\vXR}{\va{\ca X(\RR)}}

\begin{document}

\title{\Large{\textbf{Topological groupoids with involution and \\ real algebraic stacks}}}

\author{Emiliano Ambrosi and Olivier de Gaay Fortman \\ \today}

\date{}

\onehalfspacing

\maketitle 

\abstract{\noindent
\emph{To a topological groupoid endowed with an involution, we associate a topological groupoid of fixed points, generalizing the fixed-point subspace of a topological space with involution.
We prove that when the topological groupoid with involution arises from a Deligne-Mumford stack over \( \mathbb{R} \), this fixed locus coincides with the real locus of the stack. This provides a topological framework to study real algebraic stacks, and in particular real moduli spaces. Finally, we propose a Smith--Thom type conjecture in this setting, generalizing the Smith--Thom inequality for topological spaces endowed with an involution. 
}
%We study topological groupoids\blfootnote{\emph{Date:} \today.} equipped with an involution. To each such an object we associate a topological groupoid of fixed points, generalizing the fixed-point subspace of a topological space with involution. This allows us to propose a Smith--Thom type conjecture in this setting.  %of topological groupoids with an involution.
%When the topological groupoid arises from a Deligne–Mumford stack over \( \mathbb{R} \), this fixed-point locus coincides with the real locus of the stack. In this way, our framework provides tools to study the topology of real Deligne--Mumford stacks.

%To a Deligne--Mumford stack over $\RR$, one can associate a topological groupoid $\mr X$ with involution $\sigma$. For any such a pair $(\mr X, \sigma)$, we introduce a topological space of fixed points $\va{\mr X^G}$ that generalizes the fixed point subspace of a topological space; it coincides with the real locus of a real Deligne--Mumford stack. 

\tableofcontents

\section{Introduction}
This work is the first in a two-part series devoted to the study of the topology of real algebraic stacks. In this paper, which takes a more topological perspective, we study the topology of real algebraic stacks via their associated topological groupoids with involution, and establish several general results in this setting. We also formulate a Smith–Thom type conjecture for such groupoids.

The second paper of this series, see \cite{ambrosi2025topologyrealalgebraicstacks}, adopts a more algebraic point of view. There, we develop techniques for computing the topology of various real Deligne–Mumford stacks -- such as finite quotient stacks and gerbes over a real variety -- and use these techniques to verify the conjecture in several cases.
\subsection{The topology of real moduli spaces.} Our interest in topological groupoids with involution is motivated by the study of the topology of moduli spaces over \( \mathbb{R} \), see e.g.\ 
\cite{gross-harris, seppalasilhol, cohomology-M0n} and \cite{thesis-degaayfortman, degaay-realmoduli}. Recently, there has been growing interest in understanding the cohomology of the real locus $M(\RR)$ of the associated coarse moduli space $M$, see e.g.\ \cite{franz-2018, brugalleschaffhauser-2022, fu2023maximalrealvarietiesmoduli, kharlamov2024unexpectedlossmaximalitycase, kharlamov2025smiththomdeficiencyhilbertsquares}. However, such a study says something about the \emph{real moduli space} associated to the moduli problem only when this real moduli space coincides with the real locus of the coarse moduli space---a phenomenon that fails in general (think of moduli of abelian varieties or curves).

To overcome this limitation, one is naturally led to move beyond the setting of real algebraic varieties (and hence of topological spaces with involution) and into the realm of real algebraic stacks (and hence of topological groupoids with involution). Similar to the way in which purely topological methods often play a crucial role in the study of the real locus \( X(\mathbb{R}) \) of a real algebraic variety \( X \), the study of real algebraic stacks is facilitated by working within the context of topological groupoids with involution.

%\subsection{Outline of the paper.}

\begin{comment}
In the remainder of the introduction, we begin by briefly recalling the notion of a topological groupoid $\mr X$, and its connection to complex \\Deligne–Mumford stacks. We then turn to the problem of defining a suitable notion of \emph{fixed locus} of an involution $\sigma \colon \mr X \to \mr X$ defined on a topological groupoid $\mr X$, one that recovers the classical fixed locus in the case of topological spaces and the real locus in the case of a real Deligne–Mumford stack. With this notion in place, we proceed to state our main results. Theorem \ref{theorem:indepencence-topology-intro} asserts, roughly speaking, that the real analytic topology of a real Deligne–Mumford stack coincides with the topology of the fixed locus of its associated topological groupoid with involution. As an application, we prove in Theorem \ref{thm:criterionetale:intro} that if a real Deligne–Mumford stack is a gerbe over its coarse moduli space \( f \colon \mathcal{X} \to M \), then the induced map on real loci \( f_{\mathbb{R}} \colon \lvert \mathcal{X}(\mathbb{R}) \rvert \to M(\mathbb{R}) \) is open and a topological covering over its image. Finally, we state Conjecture \ref{conj:ST-top}, which asserts that the cohomology of the fixed locus of a topological stack with involution is bounded from above by the cohomology of the coarse space of the inertia groupoid.
\end{comment}

\subsection{Topological groupoids with involution.}\label{sec : introtopinv}Recall that a \emph{topological groupoid} $\mr X = [R\rightrightarrows U]$ consists of two topological spaces, $U$ (the space of objects) and $R$ (the space of arrows), and a collection of continuous maps $(s,t,c,e,i)$ satisfying a number of natural conditions (see Section \ref{sec:topstacks:top} for more details). For instance, $s \colon R \to U$ is the source map, $t \colon R \to U$ the target map, and $c \colon R \times_{s,U,t} R \to R$ the composition map. If $f\in R$ is such that $s(f)=x$ and $t(f)=y$, we write $x\xrightarrow{f}y$ and we say that $f$ is an isomorphism between $x$ and $y$. If $f,g\in R$ are such that $t(f)=s(g)$ we write $g\circ f$ for $c(g,f)$.  

Write $x\cong y$ if there exists a $f\in R$ with $x\xrightarrow{f}y$. The axioms of topological groupoids ensure that $\cong$ is an equivalence relation, so that we can consider 
$\vert \mr X\vert = U/_{\cong}$, the set of isomorphism classes of objects of $\mr X$.
We equip $\va{\mr X}$ with the quotient topology. 
\begin{example}\label{ex:stackvcomplexsgroupoidintro}
 Let $\mathcal X$ be a Deligne--Mumford stack of finite type over $\CC$. Let $U \to \ca X$ be a surjective \'etale presentation by a scheme $U$, and let $R$ be a scheme with $R \cong U \times_{\ca X} U$. The two projections $\pi_1,\pi_2 \colon (U \times_{\ca X} U)(\CC) \to U(\CC)$ define maps $s,t \colon R(\CC) \to U(\CC)$ that extend to the structure of a topological groupoid $\mr X = [R(\CC) \rightrightarrows U(\CC)]$. There is a natural bijection between
 $\va{\mr X}$ and the set $\va{\ca X(\CC)}$ of isomorphism classes of objects in $\ca X(\CC)$; see Construction \ref{cons:groupoid-DM} for more details. 
\end{example}

An \emph{involution} $\sigma \colon \mr X \to \mr X$ consists of involutions $\sigma \colon R \to R$ and $\sigma \colon U \to U$ that are compatible all the structure maps of the topological groupoid, see Section \ref{sec:topstacks:top:inv} for the precise definition.
\begin{example}\label{ex:stackvsgroupoidintro}
 Let $\mathcal X$ be a Deligne--Mumford stack of finite type over $\RR$, let $U \to \ca X$ be a surjective \'etale presentation by a scheme $U$ over $\RR$. % and let $R = U \times_{\ca X} U$.  
If $\mr X = [R(\CC)\rightrightarrows U(\CC)]$ is the topological groupoid associated to $\mathcal X_{\mathbb C}$ as in Example \ref{ex:stackvcomplexsgroupoidintro}, the natural anti-holomorphic involutions on $R(\CC)$ and $U(\CC)$ define an involution $\sigma \colon \mr X \to \mr X$. 
\end{example}

Let $\mr X = [R\rightrightarrows U]$ be a topological groupoid, equipped with an involution $\sigma \colon \mr X\rightarrow \mr X$. Inspired by the case of topological spaces with involution, we call the pair $(\mr X, \sigma)$ a \emph{topological $G$-groupoid}, where
\[
G \coloneqq \Gal(\CC/\RR) \cong \ZZ/2.
\]
We define the \emph{fixed locus} of the topological $G$-groupoid $(\mr X, \sigma)$ as follows:
\begin{equation}\label{def:real-locusintro}
\va{\mr{X}^{G}} \coloneqq \set{(x, \varphi) \in U \times R \text{ such that  $x \xrightarrow{\varphi} \sigma(x)$ with $\sigma(\varphi) \circ \varphi = \id$}}/_{\cong}
\end{equation}
where $(x, \varphi)\cong (y, \psi)$  if there exists an arrow $f \colon x \to y$ in $R$ such that $\psi \circ f = \sigma(f) \circ \varphi$. As the notation suggests, the set $\vert \mathscr X^G \vert$ actually arises as the set of isomorphism classes of the space of objects of a topological groupoid $\mathscr X^G$, see Definition \ref{def:fix groupoid} below. This definition is motivated by the following example.
 \begin{example}\label{ex:realvsgroupoidintro}
 Let $\mathcal X$ be a Deligne--Mumford stack of finite type over $\RR$, and let $\mr X = [R(\CC)\rightrightarrows U(\CC)]$ be the topological groupoid with involution associated to $\mathcal X$ as in Example \ref{ex:stackvsgroupoidintro}. By the axioms of a stack over $\RR$, which imply that the objects and morphisms of $\ca X$ satisfy descent for étale coverings (see \cite[Définition (3.1), page 15]{laumon-moretbailly}), one has a canonical bijection $\va{\mr X^G} \cong \va{\ca X(\RR)}$. For details, see Lemma \ref{lemma:realDMstackasgrupoid2}.
 \end{example}
 %Let $\ca X$ be a separated Deligne--Mumford stack of finite type over $\RR$, and consider the functor $\sigma \colon \ca X(\CC) \to \ca X(\CC)$ that sends an object $x \in \ca X(\CC)$ to its pull-back along complex conjugation $ \Spec(\CC) \to \Spec(\CC)$. Let $\kappa \colon \sigma^2 \Rightarrow \id$ be the canonical isomorphism between $\sigma^2 \colon \ca X(\CC) \to \ca X(\CC)$ and the identity functor. By Galois descent, see \cite{grothendieck-descente-I}, the set of isomorphism classes $\vert \mathcal X(\mathbb R)\vert $ of objects in $\ca X(\RR)$ is naturally in bijection the set of isomorphism classes of pairs $(x,\varphi)$, where $x\in \mathcal X(\mathbb C)\vert$ and $\varphi \colon  x \xrightarrow{\sim} \sigma(x)$ is an isomorphism such that $\kappa \circ \sigma(\varphi) \circ \varphi = \id$ as isomorphisms $x \xrightarrow{\sim} x$. 
%$\sigma_{\mathcal X}(\varphi) \circ \varphi = \id$. 

\subsection{Topology of real Deligne--Mumford stacks.} If \( \mathcal{X} \) is a real Deligne–Mumford stack, one can define a natural topology on the set \( \lvert \mathcal{X}(\mathbb{R}) \rvert \) of isomorphism classes of the groupoid $\ca X(\RR)$, by choosing an étale presentation \( U \to \mathcal{X} \) by a real scheme \( U \) such that \( U(\mathbb{R}) \to \lvert \mathcal{X}(\mathbb{R}) \rvert \) is surjective (cf.\ \cite[Theorem 7.4]{degaay-realmoduli}); the resulting topology is independent of the choice of presentation (cf.\ \cite[Proposition 7.6]{degaay-realmoduli}) and called the \emph{real analytic topology} of $\va{\ca X(\RR)}$ (cf.\ \cite[Definition 7.5]{degaay-realmoduli}). 

Our first main result compares the real analytic topology of $\va{\ca X(\RR)}$ with the topology of the associated topological groupoid. For a topological groupoid $\mr X = [R\rightrightarrows U]$ equipped with an involution $\sigma\colon\mr X\rightarrow \mr X$, endow the fixed locus \( \lvert \mr{X}^G \rvert \) with the quotient topology (it is a quotient of a subspace of $U \times R$, see equation \eqref{def:real-locusintro} above). 
\begin{theorem} \label{theorem:indepencence-topology-intro}
Let $\ca X$ be a separated Deligne--Mumford stack of finite type over $\RR$. Let $U \to \ca X$ be an \'etale surjective morphism from a scheme $U$ over $\RR$. Consider the associated topological groupoid  $\mr X = [R(\CC) \rightrightarrows U(\CC)]$ with involution $\sigma\colon \mr X\rightarrow \mr X$.  %and \( \lvert \mathcal{X}(\mathbb{R}) \rvert \) with the real analytic topology. 
When $\va{\mr X^G}$ is endowed with the quotient topology and $\va{\ca X(\RR)}$ with the real analytic topology, the bijection  
    $\va{\mr X^{G}} \xrightarrow{\sim} \va{\ca X(\RR)}$
   of Example \ref{ex:realvsgroupoidintro} is a homeomorphism.\end{theorem}
As an application of Theorem \ref{theorem:indepencence-topology-intro}, and to illustrate the value of topological groupoid techniques in the study of the topology of real algebraic stacks, we analyze the map on real loci  
$
f_{\mathbb{R}} \colon \lvert \mathcal{X}(\mathbb{R}) \rvert \to M(\mathbb{R})
$  
induced by the coarse moduli space morphism \( f \colon \mathcal{X} \to M \), where \( \mathcal{X} \) is a separated Deligne–Mumford stack of finite type over \( \mathbb{R} \) (recall that the existence of \( f \) is guaranteed by \cite{keelmori}). Our main result in this direction is as follows. 

\begin{theorem}\label{thm:criterionetale:intro}
Let $\ca X$ be a separated Deligne--Mumford stack of finite type over $\RR$, with coarse moduli space $f \colon \ca X \to M$. Assume $\# \Aut(x)$ is constant for $x \in \ca X(\CC)$. Then the map $f_\RR \colon \va{\ca X(\RR)} \to M(\RR)$ is open, and a topological covering over its image. 
\end{theorem} 
Topological groupoid techniques play a key role in the proof of Theorem \ref{thm:criterionetale:intro}, as they allow one to work “euclidean locally” on \( M(\mathbb{R}) \), in analogy with the familiar approach of working analytically locally in the study of real algebraic varieties. Theorem \ref{thm:criterionetale:intro} will be applied in the subsequent paper \cite{ambrosi2025topologyrealalgebraicstacks} to study the topology of various types of real gerbes over a real variety.
\subsection{Smith--Thom inequality for topological groupoids with involution.}
 Recall that if $T$ is a locally compact Hausdorff topological space such that $\dim \rm H^\ast(T,\ZZ/2)$ is finite, endowed with a $G$-action given by an involution $\sigma\colon T\rightarrow T$, 
 the \emph{Smith--Thom inequality} states that
\begin{align} \label{align:inequality-ST}
\dim \rm{H}^\ast(T^G,\ZZ/2) \leq \dim \rm{H}^\ast(T, \ZZ/2).
\end{align}
See \cite{floydperiodic, borel-seminar, thom-homologie, itenberg-enriques, mangolte} for various proofs.
The inequality \eqref{align:inequality-ST} is particularly significant when \( T = X(\mathbb{C}) \) for a real algebraic variety \( X \), and \( \sigma \colon X(\mathbb{C}) \to X(\mathbb{C}) \) is the anti-holomorphic involution given by complex conjugation, so that \( T^G = X(\mathbb{R}) \). In this setting, the inequality \eqref{align:inequality-ST} provides an upper bound on the cohomology of \( X(\mathbb{R}) \) in terms of the cohomology of \( X(\mathbb{C}) \), usually easier to compute. As such, it stands as one of the foundational results in real algebraic geometry.

%  \subsubsection{The Smith-Thom inequality for topological groupoids with involution}

 In light of Theorem \ref{theorem:indepencence-topology-intro}, a natural first step toward understanding the topology of real moduli spaces is to ask whether an equality like \eqref{align:inequality-ST} might hold for topological groupoids with an involution. Unfortunately, the naive inequality  
$
\dim \rm H^\ast(\lvert \mr{X}^{G} \rvert, \ZZ/2) \leq \dim \rm H^\ast(\lvert \mr{X} \rvert,\ZZ/2)
$  
fails in general for such groupoids. For example, consider the classifying groupoid \( \mr{X} \coloneqq [\Gamma \rightrightarrows \mathrm{pt}] \) attached to a finite group $\Gamma$, equipped with the trivial involution \( \sigma = \id \colon \mr{X} \to \mr{X}\). In this case, \( \lvert \mr{X} \rvert \simeq \rm{pt}\) while \( \lvert \mr{X}^{G} \rvert \simeq \rm H^1(G,\Gamma)  \) (cf.\ Example \ref{ex:fixed-BGamma}), so the inequality fails when $\#\rm H^1(G,\Gamma) > 1$ (e.g., when $\Gamma = \ZZ/2$).

  This example shows that in a sense, the topological space $\vert \mr X\vert$ is too small to fully encode information about \(\lvert \mathcal{\mr X}^{G}\rvert\), as it does not capture the automorphisms of objects in \(\mr X\). To take these into account, we consider the inertia groupoid $\ca I_{\mr X} = [S \rightrightarrows R|_\Delta]$ whose objects are given by the space $R|_\Delta$ of arrows $\varphi \in R$ such that $\varphi$ is an automorphism of $x = s(\varphi)$, and whose morphisms $\varphi \to \varphi'$ are given by morphisms $f \colon s(\varphi) \to s(\varphi')$ in $R$ such that $\varphi' \circ f = f \circ \varphi$. 

We conjecture the following generalization of the Smith–Thom inequality \eqref{align:inequality-ST} to the setting of topological groupoids with involution.

\begin{conjecture} \label{conj:ST-top}
Let $\mr X = [R \rightrightarrows U]$ be a topological groupoid with involution $\sigma\colon \mr X\rightarrow \mr X$. 
Assume that $R$ and $U$ are locally compact and Hausdorff, and that the spaces $\va{\mr X^{G}}$ and $\va{\ca I_{\mr X}}$ have finite dimensional $\ZZ/2$-cohomology. Then, we have:
\begin{align} \label{align:inequality-ST-stacks}
\dim \rm H^\ast( \va{\mr X^{G}}, \ZZ/2) \leq \dim \rm H^\ast( \va{\ca I_{\mr X}}, \ZZ/2). 
\end{align}
\end{conjecture}
When $X = \mr X$ is a topological space, the natural map $\va{ \ca I_{\mr X}}\rightarrow \va{\mr X}$ is a homeomorphism, hence (\ref{align:inequality-ST-stacks}) reduces to the usual Smith--Thom inequality (\ref{align:inequality-ST}). Moreover, the bound (if valid) is sharp; in fact, one can easily construct examples of topological groupoids with involution which are not topological spaces for which (\ref{align:inequality-ST-stacks}) is an equality. 

%the following weaker version of Conjecture \ref{conj:ST-top} in various cases.
Restricting Conjecture \ref{conj:ST-top} to those topological groupoids with involution that arise from a real Deligne--Mumford stack, one obtains the following algebraic version. 
\begin{conjecture} \label{conj:ST:algebraic}
Let $\ca X$ be a separated Deligne--Mumford stack of finite type over $\RR$. Let $I_{\ca X}$ be the coarse moduli space of the inertia stack of $\ca X$. 
Then the dimension of $\rm{H}^\ast(\va{\ca X(\RR)},\ZZ/2)$ is less than or equal to the dimension of $\rm{H}^\ast(I_{\ca X}(\CC), \ZZ/2)$. 
%the following inequality holds:
%\begin{align*}%\label{align:inequality-ST-stacks-algebraic}
%h^\ast(\va{\ca X(\RR)}) = 
%= h^\ast(\vert \ca I_{\ca X}(\CC)\vert). 
%\end{align*}
\end{conjecture}
Note that, by Theorem \ref{theorem:indepencence-topology-intro}, Conjecture \ref{conj:ST-top} implies Conjecture \ref{conj:ST:algebraic}. Providing evidence for Conjecture \ref{conj:ST:algebraic} requires the development of techniques for computing the topological space \( \lvert \mathcal{X}(\mathbb{R}) \rvert \); such techniques are currently unavailable, as the topology of real algebraic stacks has so far been an unexplored area. These methods will be developed in the second part of this two-paper series, see \cite{ambrosi2025topologyrealalgebraicstacks}. %There, we will provide evidence for Conjecture \ref{conj:ST-top}, by establishing

In the present work, we restrict ourselves to verifying Conjecture \ref{conj:ST-top} in the case of the classifying groupoid $B\Gamma$ attached to a finite group $\Gamma$ equipped with an involution. In this setting, Conjecture \ref{conj:ST-top} is equivalent to a group-theoretic statement, whose proof has been kindly communicated to us by Will Sawin, cf.\ \cite{Sawinoverflow}. The result is as follows. 

%There, we will provide evidence for Conjecture \ref{conj:ST-top}, by establishing it for various groupoids with involution arising from Deligne--Mumford stacks $\ca X$. In this case, by Theorem \ref{theorem:indepencence-topology-intro}, the inequality in the conjecture is equivalent to the inequality
%$\dim \rm H^\ast(\va{\ca X(\RR)}, \ZZ/2) \leq \dim \rm H^\ast(I_{\ca X}(\CC),\ZZ/2)$, where $I_{\ca X}$ is the coarse moduli space of the inertia stack $\ca I_{\ca X}$ of $\ca X$. %Hence, if the stack $\ca X$ represents a moduli functor of algebraic objects over $\RR$, then the conjecture predicts something about the topology of  $\va{\ca X(\RR)}$, which can be called the \emph{real moduli space} associated to the moduli problem. 

\begin{proposition}\label{propintro:STsclaffyingstack}
  Let $\Gamma$ be a finite group and $\sigma \colon \Gamma \to \Gamma$ an involution. Define $B\Gamma = [\Gamma \rightrightarrows \rm{pt}]$, endowed with the involution induced by $\sigma$.  The inequality (\ref{align:inequality-ST-stacks}) holds for $B\Gamma$. 
\end{proposition}
\subsection{Variant for groupoid cohomology.}
Our construction of the fixed locus \( \lvert \mr{X}^G \rvert \) of a topological groupoid with involution \( \mr{X} \) actually proceeds by first associating to \( \mr{X} \) a topological groupoid of fixed points \( \mr{X}^G \), see Definition \ref{def:fix groupoid}; the space \( \lvert \mr{X}^G \rvert \) is then defined as the coarse space associated to this fixed-point groupoid. From this perspective, Conjecture \ref{conj:ST-top} compares the topology of the coarse space \( \lvert \mr{X}^G \rvert \) of $\mr X^G$ with the topology of the coarse space \( \lvert \mathcal{I}_{\mr{X}} \rvert \) of the inertia groupoid of \( \mr{X} \). In a complementary direction, one could compare the \emph{groupoid cohomology} of \( \mr{X}^G \) with the groupoid cohomology of \( \mr{X} \), as an alternative route towards generalizing the Smith–Thom inequality \eqref{align:inequality-ST} to topological groupoids. However, this framework makes it more difficult to formulate a precise conjecture, since groupoid cohomology is often infinite-dimensional (think of the case \(\mr X = B\Gamma \) for \( \Gamma = \mathbb{Z}/2 \)). We conclude the paper with Section \ref{sec:smith-thom-orbifold}, where we explore possible variants of Conjecture \ref{conj:ST-top} in the context of groupoid cohomology, presenting several related questions and examples.

\subsection{Organization of the paper.}
The paper is organized as follows. In Section \ref{sec:notation}, we collect some notation and conventions. In Section  
 \ref{sec:preliminaries-top-groupoids}, we briefly review the theory of topological groupoids, define the fixed groupoid of a topological $G$-groupoid, and prove some preliminary results. Section \ref{sec:realanalytictopology} is devoted to the proof of Theorem \ref{theorem:indepencence-topology-intro}, and Section \ref{sec:top-of-gerbes} to the proof of Theorem \ref{thm:criterionetale:intro}. In Section \ref{sec:classifying}, we prove Proposition \ref{propintro:STsclaffyingstack}. We conclude the paper in Section \ref{sec:smith-thom-orbifold} with some speculative remarks on a possible analogue of the Smith–Thom inequality in the context of groupoid cohomology.

\subsection{Notation and conventions} \label{sec:notation}
We let $G$ denote the finite group
$
G \coloneqq \Gal(\CC/\RR) \cong \ZZ/2,
$
and $\sigma \in G$ a generator. A topological groupoid $(U,R, s,t,c,e,i)$, consisting of topological spaces $U$ and $R$ and continuous maps $s,t \colon R \to U$, $c \colon R \times_{s,U,t} R \to R$, $e \colon U \to R$ and $i \colon R \to R$ that satisfy the usual compatibility conditions (see \cite[\href{https://stacks.math.columbia.edu/tag/0230}{Tag 0230}]{stacks-project}), will be denoted by $[R \rightrightarrows U]$. 

A \emph{variety} over $\RR$ (resp. $\CC$) will be a reduced and separated scheme of finite type over $\RR$ (resp.\ $\CC$). For a scheme $S$, and $(\Sch/S)_{fppf}$ the big fppf site of $S$ (cf.\ \cite[\href{https://stacks.math.columbia.edu/tag/021S}{Tag 021S}]{stacks-project}), an \emph{algebraic stack} over $S$ is a stack in groupoids $p \colon \ca X \to (\Sch/S)_{fppf}$ (see \cite[\href{https://stacks.math.columbia.edu/tag/0304}{Tag 0304}]{stacks-project}) which satisfies the conditions in \cite[\href{https://stacks.math.columbia.edu/tag/026O}{Tag 026O}]{stacks-project}. 
We will indicate an algebraic stack by a calligraphic letter, such as $\ca X, \ca Y, \ca Z$. Schemes are usually indicated by roman capitals, such as $X,Y,Z$. For an algebraic stack $\ca X$, we let $\ca I_{\ca X}\to \ca X$ denote the inertia stack over $\ca X$, see e.g.\ \cite[\href{https://stacks.math.columbia.edu/tag/036X}{Tag 036X}]{stacks-project}. When $\ca X$ is an algebraic stack over a scheme $S$, we let $\va{\ca X(S)}$ denote the set of isomorphism classes of the groupoid $\ca X(S)$.

An algebraic stack $\ca X$ over $S$ is a \emph{Deligne--Mumford stack} if there exists a scheme $U$ and a surjective \'etale morphism $U \to \ca X$. We will repeatedly use the following theorem by Keel and Mori \cite{keelmori}: if $\ca X$ is a separated Deligne--Mumford stack locally of finite type over a scheme $S$, there exists a coarse moduli space $\ca X \to M$. In particular, if $S = \Spec(\RR)$, then $M$ is an algebraic space locally of finite type over $\RR$, and hence its complex locus $M(\CC)$ has naturally the structure of a complex-analytic space (see \cite[Chapter I, Proposition 5.18]{knutson}) endowed with an anti-holomorphic involution with fixed locus $M(\RR)$. This provides the real locus $M(\RR)$ of $M$ with the structure of a closed real analytic subspace $M(\RR) \subset M(\CC)$ and in particular with a topology.

 %For an object $x \in \ca X(S)$, we let $[x] \in \va{\ca X(S)}$ denote its isomorphism class. 

\subsection{Acknowledgements.} We thank Will Sawin for explaining to us the proof of Lemma \ref{lem:sawin}. We thank Olivier Benoist, Ilia Itenberg, Matilde Manzaroli, Ieke Moerdijk and Florent Schaffhauser for helpful discussions.  We thank the anonymous referees for their thoughtful comments.
This research was partly supported by the grant ANR–23–CE40–0011 of Agence National de la Recherche.
The second author has received funding %from the ERC Consolidator Grant FourSurf N\textsuperscript{\underline{o}}101087365.
 from the European Research Council (ERC) under the European Union’s Horizon 2020 research and innovation programme under grant agreement N\textsuperscript{\underline{o}}948066 (ERC-StG RationAlgic) and from the ERC Consolidator Grant FourSurf N\textsuperscript{\underline{o}}101087365.

\section{Topological groupoids with involution} \label{sec:preliminaries-top-groupoids}
In this section, we associate a groupoid of fixed points $\mr X^G$ to a topological groupoid $\mr X$ endowed with an involution $\sigma \colon \mr X \to \mr X$, and establish preliminary results concerning such groupoids. %Section \ref{sec:preliminaries-top-groupoids} is organized as follows. 
In Section \ref{sec:topstacks:top}, we recall the definition of topological groupoids and examine basic properties of their associated coarse spaces. In Section \ref{sec:topstacks:top:inv}, we introduce groupoids of fixed points, and provide examples. Finally, in Section \ref{sec:propfixpoint}, we prove structural properties of the groupoid of fixed points that will be used in later sections.

\subsection{Topological groupoids.} \label{sec:topstacks:top}
Recall that a topological groupoid is a groupoid object in the category of topological spaces (see e.g.\ \cite[Section 1]{crainic-moerdijk}). As such, a topological groupoid $\mr X = [R\rightrightarrows U]$ consist of two topological spaces, $U$ (the space of objects) and $R$ (the space of arrows), and a collection of continuous maps $s \colon R \to U$ (source), $t \colon R \to U$ (target), $c \colon R \times_{U} R \to R$ (composition), $e \colon U \to R$ (identity) and $i \colon R \to R$ (inversion), satisfying various natural compatibilities (see e.g.~\cite[\href{https://stacks.math.columbia.edu/tag/0230}{Tag 0230}]{stacks-project}). Recall the notation $U/_{\cong}$ from Section \ref{sec : introtopinv}.
%(see e.g. \cite[Section 2]{gursoy})
\begin{definition}
    One defines $\vert \mr X\vert = U/R = U/_{\cong}$, the set of isomorphism classes of objects $x \in U$, and one equips $\va{\mr X}$ with the quotient topology induced by $U \to \va{\mr X}$. The topological space $\va{\mr X}$ is called the \emph{coarse space} of the topological groupoid $\mr X$. 
\end{definition}
 For later use, we make the following observation. 
\begin{lemma} \label{lem:quotient-open}
   Assume that the target map $t \colon R \to U$ is open. Then the quotient map $\pi \colon U \to \va{\mr X} = U/R$ is open.
\end{lemma}
\begin{proof}
    For $B \subset U$ open, we have that $\pi^{-1}(\pi(B)) = t(s^{-1}(B))$ is open in $U$. Hence $\pi(B)$ is open in $U/R$. 
\end{proof}
In the classical Smith--Thom inequality \eqref{align:inequality-ST}, one assumes that the topological spaces are locally compact Hausdorff, see \cite{borel-seminar}. We give a criterion to guarantee that the coarse space of a topological groupoid is locally compact Hausdorff.
\begin{lemma}
    Let $\mr X = [R \rightrightarrows U]$ be a topological groupoid. Assume that $U$ and $R$ are locally compact Hausdorff. Suppose the quotient map $\pi \colon U \to \va{\mr X} = U/R$ is open and the map $(s,t) \colon R \to U \times U$ is closed. Then $\va{\mr X} = U/R$ is locally compact Hausdorff. 
\end{lemma}
\begin{proof}
    Let $Z \subset U \times U$ be the image of $(s,t) \colon R \to U \times U$. Let $\sim$ be the equivalence relation on $U$ such that $x \sim y$ if and only if $x=y$ or $(x,y) \in Z$. 
    Since $U$ is Hausdorff, $\pi \colon U \to U/_\sim = U/R$ is open and $Z \subset U \times U$ is closed, we have that $\va{\mr X} = U/_\sim$ is Hausdorff. Since $\pi$ is open and $U$ is locally compact, $\va{\mr X}$ is locally compact. 
    %For $[x] \in U/R$, let $x \in \pi^{-1}([x])$ and let $x \in U \subset K$ with $U$ open and $K$ compact. Then $[x] \in \pi(U) \subset \pi(K)$ with $\pi(U)$ open and $\pi(K)$ compact. 
\end{proof}

\subsection{Topological groupoids with involution.} \label{sec:topstacks:top:inv}
Recall that if  \( \mr{X}_1 = [R_1 \rightrightarrows U_1] \) and  \( \mr{X}_2 = [R_2 \rightrightarrows U_2] \) are two topological groupoids, then a morphism of topological groupoids $F\colon \mr{X}_1  \rightarrow \mr{X}_2 $ is a pair $(f,g)$ of continuous maps $f\colon R_1\rightarrow R_2$ and $g\colon U_1\rightarrow U_2$ compatible with all the structure maps of the groupoids. A \emph{topological $G$-groupoid} (or a \emph{topological groupoid with involution}) is a pair ($\mr X$, $\sigma$), where $\mr X$ is a topological groupoid and $\sigma\colon \mr X\rightarrow \mr X$ is an involution (i.e., a morphism that satisfies $\sigma^2=\id$). For topological $G$-groupoids $\mr X_1$ and $\mr X_2$, a morphism $F = (f,g) \colon \mr X_1 \to \mr X_2$ is called \emph{$G$-equivariant} if $f \colon R_1 \to R_2$ and $g \colon U_1 \to U_2$ commute with the involutions on both sides. 

Let \( \mr{X} = [R \rightrightarrows U] \) be a topological $G$-groupoid. We now define the appropriate analogue of the fixed locus of the involution (compare \cite{Romagny}). 

\begin{definition} \label{def:fix groupoid} 
The \emph{groupoid of fixed points} is the topological groupoid $\mr X^G \coloneqq  [\rm A^1(G,R) \rightrightarrows \rm Z^1(G,R)]$, where:
\begin{itemize}
    \item $\rm Z^1(G,R)$ is the subspace $\rm Z^1(G, R)\subset U \times R$ of pairs $(x, \varphi)$ such that $\varphi$ is an isomorphism $x \xrightarrow{\sim} \sigma(x)$ such that $\sigma(\varphi) \circ \varphi = \id$; in particular we have two natural maps $\rm Z^1(G,R) \to U, (x, \varphi)\mapsto x$ and $\rm Z^1(G,R) \to R, (x, \varphi) \mapsto \varphi$.
    \item $\rm A^1(G,R)$ is the fibre product $\rm A^1(G,R) = R \times_{s,U} \rm Z^1(G,R)$; more explicitly, an arrow $f \colon (x, \varphi) \to (x', \varphi')$ between two objects $(x,\varphi), (x', \varphi') \in \rm Z^1(G,R)$ is given by an arrow $f \in R$ with $s(f) = x$ and $t(f) = x'$, such that $\sigma(f) \circ \varphi = \varphi' \circ f$ as maps $x \to \sigma(x')$.
    \item The maps $s, t \colon \rm A^1(G,R) \to \rm Z^1(G,R)$ are defined as $s(f, (x, \varphi)) = (x, \varphi) = (s(f), \varphi)$ and $t(f, (x, \varphi)) = (t(x), \sigma(f) \circ \varphi \circ f^{-1})$.
    \item  The inversion map $i \colon \rm A^1(G,R) \to \rm A^1(G,R)$ is defined by sending an arrow $f \colon (x, \varphi) \to (x', \varphi')$ with $f \in R$ to the arrow $f^{-1} \in R$. The composition of two arrows is induced by the composition in $R$, and finally, the identity $e \colon \rm Z^1(G,R) \to \rm A^1(G,R)$ is defined by sending $(x, \varphi)$ to the identity $\id_x \colon (x, \varphi) \to (x, \varphi)$, where $\id_x \in R$ is the identity arrow of $x \in U$.
\end{itemize}
 With these arrows $(s, t, c, e, i)$, one can check that $\mr X^G =  [\rm A^1(G,R) \rightrightarrows \rm Z^1(G,R)]$ is indeed a topological groupoid. 
\end{definition}

%\end{definition}
%\begin{definition} \label{def:real-locus}
%    Let $\mr X = [R \rightrightarrows U]$ be a topological $G$-groupoid. 
%Define $\rm Z^1(G, R)\subset U \times R$ 
%as the subspace of pairs $(x, \varphi) \in U \times R$ such that $\varphi$ is an isomorphism $x \xrightarrow{\sim} \sigma(x)$ with $\sigma(\varphi) \circ \varphi = \id$. Then $\rm Z^1(G, R)$ is the set of objects of a topological groupoid $\mr X^G = [\rm A^1(G,R) \rightrightarrows \rm Z^1(G,R)]$, whose arrows between $(x,\varphi) \in \rm Z^1(G, R)$ and $(y, \psi) \in \rm Z^1(G, R)$ are given by isomorphisms $f \colon x \to y$ in $R$ such that $\psi \circ f = \sigma(f) \circ \varphi$.
%In particular, we can consider the underlying topological space $\va{\mr X^G}$ of $\mr X^G$, defined as
%\[
%\va{\mr X^G} = \rm Z^1(G,R)/_{\cong} 
%\]
%equipped with the quotient topology induced by $\rm Z^1(G,R) \to \va{\mr %X^G}$.??? 
%\end{definition}

\begin{example} \label{ex:DM-top-stack}
Let $\mr T$ be a stack on the site of all topological spaces. Define an \emph{involution} on $\mr T$ to be a $1$-morphism $\sigma \colon \mr T \to \mr T$ such that $\sigma^2$ is $2$-isomorphic to the identity functor (where $2$-isomorphic should be taken in the sense of \cite[\href{https://stacks.math.columbia.edu/tag/02ZK}{Tag 02ZK}]{stacks-project}). Say that a  \emph{topological $G$-stack} consists of a topological stack $\mr T$ in the sense of \cite[Section 2]{noohi-2012} and an involution $\sigma$ on $\mr T$, such that there exists a representable surjective morphism $p \colon U \to \mr T$ where $U$ is a topological space equipped with an involution $\sigma_U \colon U \to U$ that $2$-commutes with $\sigma$ and $p$. %Let $(\mr T, \sigma)$ be a topological $G$-stack. 
For such a stack $\mr T$ with presentation $p \colon U \to \mr T$, let $R$ be a topological space with $R \cong U \times_{\mr T}U$. Then $R$ carries a natural involution $\sigma_R \colon R \to R$ induced by $\sigma$ and $\sigma_U$, %defined by the involution $(x,y, \phi) \mapsto (\sigma_U(x), \sigma_U(y), \sigma(\phi))$ on $U \times_{\mr T}U$, %where $\sigma(\phi)$ is the map $p(\sigma_U x) \to p(\sigma_U y)$ induced by the map $\sigma(px) \to \sigma(py)$. 
and the involutions $\sigma_U$ and $\sigma_R$ induce on $\mr X = [R \rightrightarrows U]$ the structure of a topological $G$-groupoid.
\end{example}

\begin{example} \label{ex:fixed-BGamma}
If $\Gamma$ is a finite group equipped with an involution $\sigma \colon \Gamma \to \Gamma$, and if $\mr X = B\Gamma = [\Gamma \rightrightarrows \rm{pt}]$ is the classifying groupoid of $\Gamma$, then there is a natural isomorphism of topological groupoids
    $$
    \mr X^G = \coprod_{[\gamma] \in \rm H^1(G, \Gamma)} B(\Gamma^{\sigma_\gamma})
    $$
    where, for an element $\gamma \in \Gamma$ with $\gamma \sigma(\gamma) = e$, the subgroup $\Gamma^{\sigma_\gamma}\subset \Gamma$ consists of those elements $g\in \Gamma$ such that $\gamma\sigma(g)=g\gamma$. 
\end{example}
\subsection{Properties of the groupoid of fixed points.}\label{sec:propfixpoint}
In order to prove Theorem \ref{theorem:indepencence-topology-intro}, we need Proposition \ref{proposition:topological-indepence-Ggroupoids} below, which is the main result of this section. %We are almost ready to prove the main result of Section \ref{sec:preliminaries-top-groupoids}, which is Proposition \ref{proposition:topological-indepence-Ggroupoids} below. 
This proposition roughly says that the formation of $\va{\mr X^G}$ is invariant under a base change $U' \to U$ with suitable properties. To prove it, we need a definition and two lemmas.

\begin{definition}\label{def:open-equivalence}(Compare \cite[Section 1.5]{crainic-moerdijk}.) Let $\mr X_1 = [R_1 \rightrightarrows U_1]$ and $\mr X_2 = [R_2 \rightrightarrows U_2]$ be topological groupoids, 
and let $F\colon \mr X_2 \to \mr X_1$ be a morphism of topological groupoids, defined by a pair of compatible maps $(f \colon U_2 \to U_1, g \colon R_2 \to R_1)$.
%Assume that the maps $s,t \colon R_i \to U_i$ are surjective and open for $i = 1,2$. 
We say that $F$ is an \emph{open equivalence} if the following conditions hold: 
\begin{enumerate}
    \item the map $f \colon U_2 \to U_1$ is open,
    \item the composition
\begin{align} \label{align:comp-surj-homeo-def}
R_1 \times_{s,U_1}U_2 \to R_1 \xrightarrow{t}U_1
\end{align}
is open and surjective,
\item  the diagram 
\begin{align}\label{align:diagram:zero:H1A1R1-def}
\begin{split}
\xymatrix{
R_2 \ar[d]^{(s,t)}\ar[r]^{g} & R_1 \ar[d]^{(s,t)} \\
U_2 \times U_2 \ar[r]^{f \times f} & U_1 \times U_1
}
\end{split}
\end{align}
is cartesian.
\end{enumerate}
\end{definition}
\begin{example}
Let $X$ be a topological space and $\mr X_1:=[X\rightrightarrows X]$ be its associated groupoid, in which every arrow is the identity arrow. Let $\{U_i\}_{i\in I}$ be an open cover of $X$ and, for $i\neq j$, define $U_{i,j} \coloneqq U_i\cap U_j$. Define the topological groupoid \(\mr X_2 = [R_2 \rightrightarrows U_2]\) associated to the open cover $\{U_i\}_{i\in I}$ as follows. Put $U_2 \coloneqq  \coprod_{i \in I} U_i$ and $R_2 \coloneqq \coprod_{i,j \in I} U_{ij}$, and define the source and target arrows $s,t \colon R_2 \rightrightarrows U_2$ as the respective inclusions $U_{ij} \hookrightarrow U_i$ and $U_{ij} \hookrightarrow U_j$.
Then the natural map $\mr X_2\rightarrow \mr X_1$ is an open equivalence. 
\end{example}

\begin{example}
    Let $X$ be a locally compact Hausdorff topological space and $\Gamma$ a locally compact Hausdorff topological group. Assume that $\Gamma$ acts properly and continuously on $X$. Let $H \subset \Gamma$ be a normal subgroup acting freely on $X$, and define $X' \coloneqq X/H$ and $\Gamma' \coloneqq \Gamma/H$. Then the natural morphism $[X/\Gamma] \to [X'/\Gamma']$ is an open equivalence. Here, $[X/\Gamma]$ and $[X'/\Gamma']$ are the translation groupoids associated to the group actions. 
\end{example}

\begin{lemma} \label{lemma:open-equivalence-homeomoprhism}
    Let $F \colon \mr X_2 \to \mr X_1$ be a morphism of topological groupoids $\mr X_i = [R_i \rightrightarrows U_i]$ for which $s,t \colon R_i \to U_i$ are surjective and open ($i = 1,2$), and assume that $F$ is an open equivalence. Then the induced map $\va{F} \colon \va{\mr X_2} \to \va{\mr X_1}$ is a homeomorphism. 
\end{lemma}
\begin{proof}
Note first that $\va{F}$ is bijective. Indeed, the surjectivity of \eqref{align:comp-surj-homeo-def} implies that each $x \in U_1$ is isomorphic to an object in the image of $f \colon U_2 \to U_1$, hence $\va{F}$ is surjective; the injectivity of $\va{F}$ follows from from the fact that diagram \eqref{align:diagram:zero:H1A1R1-def} is cartesian.
%which implies that for each $a,b \in U_2$ and each isomorphism $\varphi \colon f(a) \cong f(b)$ with $\varphi \in R_1$, there exists a unique isomorphism $\psi \colon a \cong b$ with $\psi \in R_2$ such that $g(\psi)= \varphi$.

The map $\va{F}$ is continuous (since $\va{F}$ lifts to a continuous map $U_2 \to U_1$) hence it remains to prove that $\va{F}$ is open. This holds in view of the commutative diagram 
\[
\xymatrix{
U_1 \ar[r]^f \ar@{->>}[d]\ar[r] & U_2 \ar@{->>}[d] \\
U_1/R_1 \ar[r]^{\va{F}} & U_2/R_2,
}
\]
since the vertical arrows are open by Lemma \ref{lem:quotient-open} and the upper horizontal arrow is open by assumption. 
\end{proof}

\begin{lemma} \label{lemma:st-open}
    Let $\mr X = [R \rightrightarrows U]$ be a topological $G$-groupoid. Assume that the maps $s, t\colon R \to U$ are surjective and open. Then the maps $s,t\colon \rm A^1(G,R) \to \rm Z^1(G,R)$ are surjective and open. Similarly, if $s,t \colon R \to U$ are local homeomorphisms, the same holds for $s,t \colon \rm A^1(G,R) \to \rm Z^1(G,R)$.
\end{lemma}
\begin{proof}
%By definition of $\rm A^1(G,R)$, we have 
%\[
%\rm A^1(G,R) = R \times_{s,U} \rm Z^1(G,R);
%\]
Since the inversion $i \colon \rm A^1(G,R)  \xrightarrow{\sim} \rm A^1(G,R)$ is a homeomorphism, and since $s \circ i = t$ as maps $\rm A^1(G,R) \to \rm Z^1(G,R)$,  the assertions for $t$ follow from the assertions for $s$. The assertions for $s$ follow from the fact that the following diagram is cartesian:
\[
\xymatrix{
\rm A^1(G,R) \ar[d] \ar[r]^-s & \rm Z^1(G,R) \ar[d] \\
R \ar[r]^-s & U.
}
\]
This ends the proof. 
\end{proof}

\begin{proposition} \label{proposition:topological-indepence-Ggroupoids}
Let $\mr X_1 = [R_1 \rightrightarrows U_1]$ and $\mr X_2 = [R_2 \rightrightarrows U_2]$ be topological $G$-groupoids, such that the maps $s,t \colon R_i \to U_i$ are surjective and open for $i = 1,2$. Let $F\colon \mr X_2 \to \mr X_1$ be a $G$-equivariant open equivalence. 
Then the morphism $$F^G \colon \mr X_2^G \to \mr X_1^G$$ is an open equivalence. In particular, the map
$\va{F^G} \colon \va{\mr X_2^G}\to \va{\mr X_1^G}
$
is a homeomorphism. 
\end{proposition}
\begin{proof}
Let $f\colon U_2\rightarrow U_1$ and $g\colon R_2\rightarrow R_1$ be the maps corresponding to $F$. 
By Lemma \ref{lemma:open-equivalence-homeomoprhism}, it suffices to prove that the map $F^G \colon \mr X_2^G \to \mr X_1^G$ is an open equivalence. In order to do this, by Definition \ref{def:open-equivalence}, we need to check the following statements: 
\begin{enumerate}
\item[(i)] The natural map $Z^
1(G, R_2)\rightarrow Z^1
(G, R_1)$ is open.
    \item[(ii)] The natural map $\rm A^1(G,R_1) \times_{s,\rm Z^1(G,R_1)} \rm Z^1(G,R_2)\rightarrow \rm A^1(G,R_1) \xrightarrow{t} \rm Z^1(G,R_1)$ is open and surjective.
\item[(iii)] The diagram
\begin{align*}%\label{align:diagram:carteisna:H1A1R1-proof}
\begin{split}
\xymatrix{
\rm A^1(G,R_2) \ar[d]^{(s,t)}\ar[r] & \rm A^1(G,R_1) \ar[d]^{(s,t)} \\
\rm Z^1(G,R_2) \times \rm Z^1(G,R_2) \ar[r] & \rm Z^1(G,R_1) \times \rm Z^1(G,R_1)
}
\end{split}
\end{align*}
is cartesian.
\end{enumerate}

Note that $s,t \colon \rm A^1(G,R_i) \to \rm Z^1(G,R_i)$ are surjective and open by Lemma \ref{lemma:st-open}. By construction, the map $\rm Z^1(G,R_i) \to R_i$ defined as $(x, \varphi) \mapsto \varphi$ is an embedding. For $i = 1,2$, consider the restriction $(s,t) \colon \rm Z^1(G,R_i) \to U_i \times U_i$ of the map $(s,t) \colon R_i \to U_i \times U_i$ to $\rm Z^1(G,R_i)$. Observe that the diagram
\[
\xymatrix{
\rm Z^1(G,R_2) \ar[d] \ar[r]^{(s,t)} & U_2 \times U_2 \ar[d]^{f \times f} \\
\rm Z^1(G, R_1) \ar[r]^{(s,t)} & U_1 \times U_1
}
\]
is cartesian. As $U_2 \times U_2 \to U_1 \times U_1$ is open, it follows that the map $\rm Z^1(G,R_2) \to \rm Z^1(G,R_1)$ is open, hence (i) is verified.

Next, %we claim that the composition
%\begin{align} %\label{align:comp-surj-homeo-fixed-proof}
%\rm A^1(G,R_1) \times_{s,\rm Z^1(G,R_1)}\rm Z^1(G,R_2) \to \rm A^1(G,R_1) \xrightarrow{t}\rm Z^1(G,R_1)
%\end{align}
%is surjective. To prove this, 
observe that the squares in the following diagram are cartesian:
\[
\xymatrix{
\rm A^1(G,R_1) \times_{s,\rm Z^1(G,R_1)} \rm Z^1(G,R_2)\ar[d] \ar[r] & \rm A^1(G,R_1) \ar[r]^-t \ar[d]& \rm Z^1(G,R_1)\ar[d] \\
R_1 \times_{s,U_1} U_2 \ar[r] & R_1 \ar[r]^-t & U_1.
}
\]
Since the bottom row of
each square in the diagram above is surjective, the top row in the diagram is surjective as well, hence also (ii) is verified.

It remains to show (iii), which follows directly from the fact that \eqref{align:diagram:zero:H1A1R1-def} is cartesian.
\end{proof}
    
\section{Topological groupoids and Deligne--Mumford stacks} \label{sec:topgrp-DM}
The goal of this section is to prove Theorem \ref{theorem:indepencence-topology-intro}. To that end, in Section \ref{sec:topstacks}, we recall the relationship between topological groupoids and complex Deligne–Mumford stacks. In Section \ref{sec:topstacksrel}, we explain how to associate a topological groupoid with involution to a real Deligne–Mumford stack. Section \ref{sec:realanalytictopology} is devoted to recalling the definition of the real analytic topology on a real Deligne–Mumford stack, to stating Theorem \ref{theorem:indepencence-topology} -- which is a reformulation of Theorem \ref{theorem:indepencence-topology-intro} in the introduction -- and to proving this theorem by using the preliminary results from Section \ref{sec:preliminaries-top-groupoids}.

\subsection{Topological groupoids arising from complex Deligne-Mumford stacks.} \label{sec:topstacks}
For a Deligne--Mumford stack $\ca X$ locally of finite type over $\CC$, we view the set of isomorphism classes $\va{\ca X(\CC)}$ of the groupoid $\ca X(\CC)$ as a topological space, by equipping it with the quotient topology induced by the surjective morphism $U(\CC) \to \va{\ca X(\CC)}$, where $U$ is a scheme and $U \to \ca X$ a surjective \'etale morphism. This topology on $\va{\ca X(\CC)}$ does not depend on the choice of \'etale presentation $U \to \ca X$.  Every complex Deligne--Mumford stack gives rise to a topological groupoid in the following way. 

\begin{construction} \label{cons:groupoid-DM}
Let $\mathcal X$ be a complex Deligne--Mumford stack, so that there exists an \'etale surjective presentation $\pi \colon U\rightarrow \mathcal X$ by a scheme. Let $R$ be a scheme with $R \cong U \times_{\ca X} U$, so that the two projection maps $U \times_{\ca X} U \to U$ yield two maps $R \to U$ that turn $[R \rightrightarrows U]$ into a groupoid scheme. Then $\mr X = [R(\CC) \rightrightarrows U(\CC)]$ is a topological groupoid. %with objects $U(\CC)$ and arrows $R(\CC)$

Via the bijection $R(\CC) \cong \left( U \times_{\ca X} U \right)(\CC)$, any $r \in R(\CC)$ corresponds to an element $(x,y,\alpha) \in (U \times_{\ca X}U)(\CC)$ consisting of two elements $x, y \in U(\CC)$ and an isomorphism $\alpha \colon \pi(x) \xrightarrow{\sim} \pi(y)$. This yields a functor $F \colon \mr X \to \ca X(\CC)$ that sends an object $x \in U(\CC)$ to the object $\pi(x) \in \ca X(\CC)$ and an arrow $r = (x,y,\alpha)  \in R(\CC)$ to the isomorphism $\alpha \colon \pi(x) \xrightarrow{\sim} \pi(y)$. The functor $F \colon \mr X \to \ca X(\CC)$ is an equivalence of categories. 
\end{construction}
Let $\ca X$ be a Deligne--Mumford stack of finite type over $\CC$. If $\ca X$ is separated, then it has a coarse moduli space $\ca X \to M$ by \cite{keelmori}. By Construction \ref{cons:groupoid-DM} and Lemma \ref{lem:quotient-open}, it follows that the map $\va{\ca X(\CC)} \to M(\CC)$ is a homeomorphism (this can also been seen by using that $\ca X$ is \'etale locally over $M$ a finite quotient stack, see \cite[Lemma 2.2.3]{abramvistoli-2002}). 

\subsection{Topological groupoids arising from real Deligne--Mumford stacks.} \label{sec:topstacksrel}

We will often make use of the following definition.

\begin{definition}
A \emph{real DM stack} is a separated Deligne--Mumford stack of finite type over $\RR$. 
\end{definition}
Every real DM stack $\ca X$ gives rise to a topological $G$-groupoid $(\mr X, \sigma)$ in the following way. \begin{construction}\label{ex:realDMstackasgrupoid}
Let $\mathcal X$ be a real DM stack, and choose a scheme $U$ over $\RR$ and an \'etale surjective morphism $U \to \ca X$. Let $R$ be a scheme with $R \cong U \times_{\ca X} U$, so that we get a groupoid scheme $[R \rightrightarrows U]$, see Construction \ref{cons:groupoid-DM}.  
Since $U$ and $R$ are schemes locally of finite type over $\RR$, $U(\mathbb C)$ and $R(\mathbb C)$ admit natural anti-holomorphic involutions $\sigma \colon U(\CC) \to U(\CC)$ and $\sigma \colon R(\CC) \to R(\CC)$,  compatible with the structure maps of the groupoid. 
Hence  $[R(\mathbb C)\rightrightarrows U(\mathbb C)]$ is a topological groupoid with involution. 
\end{construction}

 Let $\mr X = [R(\CC) \rightrightarrows U(\CC)]$ be the topological groupoid with involution associated to a real Deligne-Mumford stack $\ca X = [U/R]$ as in Construction \ref{ex:realDMstackasgrupoid}. We aim to construct a natural equivalence of categories $F \colon \mr X^G \to \ca X(\RR)$. 
 
 To do this, recall that any $r \in R(\CC) =  \left(U \times_{\ca X}U\right)(\CC)$ corresponds to a triple $r = (x, y, \alpha)$ with $x ,y \in U(\CC)$ and $\alpha \colon \pi(x) \xrightarrow{\sim} \pi(y)$ an isomorphism, where $\pi$ is the map $U \to \ca X$. Thus, we have $$\rm Z^1(G, R(\CC)) = \set{\omega = (x, (x, \sigma(x), \varphi)) \in U(\CC) \times R(\CC) \mid \sigma(\varphi) \circ \varphi = \id}.$$ For $\omega = (x, (x, \sigma(x), \varphi)) \in \rm Z^1(G, R(\CC))$, we get an element $\pi(x) \in \ca X(\CC)$ and an isomorphism $\varphi \colon  \pi(x) \xrightarrow{\sim} \pi(\sigma(x))$ such that $\sigma(\varphi) \circ \varphi = \id$. Since, by definition, the objects of $\ca X$ satisfy descent for étale coverings (see \cite[Définition (3.1), page 15]{laumon-moretbailly}), %one has a canonical bijection $\va{\mr X^G} \cong \va{\ca X(\RR)}$.
 %By Galois descent (cf.\ \cite{grothendieck-descente-I}), 
 this yields an object $F(\omega) \in \ca X(\RR)$. Similarly, any arrow  $f \colon \omega \to \omega'$, $f \in \rm A^1(G,R(\CC))$, is given by an arrow $f = (x,x',\alpha) \in R(\CC)$ such that $\varphi' \circ \alpha = \sigma(\alpha) \circ \varphi$ as maps $\pi(x) \to \pi(\sigma(x'))$, and this yields an arrow $F(f) \colon F(\omega) \to F(\omega')$ in $\ca X(\RR)$, again by the stack axioms (see \cite[Définition (3.1), page 15]{laumon-moretbailly}). 
 This gives a natural functor 
 $$F \colon \mr X^G \to \ca X(\RR).$$
 \begin{lemma}\label{lemma:realDMstackasgrupoid2} The functor $F \colon \mr X^G \to \ca X(\RR)$ is an equivalence of categories. In particular, we get a bijection $\va{F} \colon \va{\mr X^G} \xrightarrow{\sim} \va{\ca X(\RR)}$. \end{lemma}
\begin{proof}
Since $\ca X$ is a stack over $\RR$, the objects and morphisms of $\ca X$ satisfy descent for étale coverings (see \cite[Définition (3.1), page 15]{laumon-moretbailly}); the lemma follows. 
\end{proof}

\begin{proposition}\label{prop:independence}
    Let $\ca X$ be a real DM stack. Let $U$ be a scheme with a surjective \'etale morphism $U \to \ca X$, consider the resulting topological $G$-groupoid $\mr X = [R(\CC) \rightrightarrows U(\CC)]$ (cf.\ Construction \ref{ex:realDMstackasgrupoid}). Then the topology on $\va{\ca X(\RR)}$ induced by the topology of $\va{\mr X^G}$ and the bijection $\va{F} \colon \va{\mr X^G} \xrightarrow{\sim} \va{\ca X(\RR)}$ of Lemma \ref{lemma:realDMstackasgrupoid2} does not depend on the choice of the surjective \'etale presentation $U \to \ca X$. 
\end{proposition}
\begin{proof}
    Let $U_i \to \ca X$ ($i = 1,2$) be two surjective \'etale morphisms with $U_i$ schemes, and let $R_i$ be schemes with $R_i \cong U_i \times_{\ca X} U_i$. Let $U_3$ be a scheme with an isomorphism $U_3 \cong U_1 \times_{\ca X}U_2$. Let $R_3$ be a scheme with an isomorphism $R_3 \cong U_3 \times_{\ca X} U_3$. By Construction \ref{ex:realDMstackasgrupoid}, this defines three topological $G$-groupoids 
$
\mr X_i \coloneqq [R_i(\CC) \rightrightarrows U_i(\CC)]$ $(i = 1,2,3)$. Then, for $i = 1,2$, we obtain a $G$-equivariant morphism of topological $G$-groupoids 
$
\Phi_i \colon \mr X_3 \to \mr X_i.
$
We claim that for $i = 1,2$, the morphism $\Phi_i$ is an open equivalence (in the sense of Definition \ref{def:open-equivalence}). To prove this, note that %\begin{align*}
%\left(U_3 \times_\RR U_3\right) \times_{(U_1 \times_\RR U_1)} R_1
%\cong
%\left(U_3 \times_\RR U_3\right) \times_{(U_1 \times_\RR U_1)} \left(U_1 \times_{\ca X} U_1 \right) 
%\cong
%U_3 \times_{\ca X}U_3 \cong R_3.
%\end{align*}
$R_3 \cong \left(U_3 \times_\RR U_3\right) \times_{(U_1 \times_\RR U_1)} R_1$
and $R_3 \cong \left(U_3 \times_\RR U_3\right) \times_{(U_2 \times_\RR U_2)} R_2$. 
Moreover, since $U_3 \to U_i$ is \'etale for $i =1,2$, the resulting maps $U_3(\CC) \to U_1(\CC)$ and $U_3(\CC) \to U_2(\CC)$ are local homeomorphisms (see \cite[Exposé XII, Proposition 3.1 \& Remarque 3.3]{SGA1}). In particular, the maps $U_3(\CC) \to U_i(\CC)$ are open for $i = 1,2$. As the map $U_3(\CC) \to U_i(\CC)$ is surjective for $i =1,2$, the composition $R_i(\CC) \times_{s,U_i(\CC)}U_3(\CC) \to R_i(\CC) \xrightarrow{t}U_i(\CC)
$ is surjective for $i = 1,2$, proving the claim.

Consequently, by Proposition \ref{proposition:topological-indepence-Ggroupoids}, the induced maps 
$
\va{\mr X_1^G} \leftarrow \va{\mr X_3^G} \rightarrow \va{\mr X_2^G}
$
are homeomorphisms. This proves the proposition.
\end{proof}

\subsection{The real analytic topology of a real DM stack.}\label{sec:realanalytictopology} For a real DM stack $\ca X$, the set of isomorphism classes $\va{\ca X(\RR)}$ of its real locus $\ca X(\RR)$ has a natural topology, generalizing the euclidean topology on $X(\RR)$ when $X$ is a scheme. Indeed, we have the following theorem. 
\begin{theorem} \label{theorem:DGF-reallocusstack}
Let $\ca X$ be a real DM stack. There exists a scheme $U$ over $\RR$ and a surjective \'etale morphism $U \to \ca X$ such that $U(\RR) \to \va{\ca X(\RR)}$ is surjective. 
\end{theorem}
\begin{proof}
See \cite[Theorem 2.9]{thesis-degaayfortman} or \cite[Theorem 7.4]{degaay-realmoduli}.  
\end{proof}
\begin{definition}(cf.\ \cite[Definition 7.5]{degaay-realmoduli}) \label{def:real-an-top}
    Let $\ca X$ be a real DM stack. The \emph{real analytic topology} on $\va{\ca X(\RR)}$ is defined as follows. Choose a scheme $U$ over $\RR$ and a surjective \'etale morphism $U \to \ca X$ such that $U(\RR) \to \va{\ca X(\RR)}$ is surjective. Then consider the real analytic topology on $U(\RR)$, and give $\va{\ca X(\RR)}$ the quotient topology induced by the surjection $U(\RR) \to \va{\ca X(\RR)}$. 
\end{definition}

One shows that the real analytic topology is independent of the choice of an étale presentation that is essentially surjective on real points, see \cite[Proposition 7.6]{degaay-realmoduli}.
Let now $\mathcal X$ be a real DM stack and fix a surjective \'etale morphism $U \to \ca X$. % such that $U(\RR) \to \va{\ca X(\RR)}$ is surjective (cf.\ Theorem \ref{theorem:DGF-reallocusstack}). 
Let $\mr X = [R(\CC) \rightrightarrows U(\CC)]$ be the topological groupoid with involution associated to a real Deligne-Mumford stack $\ca X = [U/R]$ as in Construction \ref{ex:realDMstackasgrupoid}. %The equivalence of categories of Lemma \ref{lemma:realDMstackasgrupoid2} induces a natural bijection of topological spaces
%$$\va{F} \colon \va{\mr X^G} \xrightarrow{\sim} \va{\ca X(\RR)}.$$
\begin{theorem}\label{theorem:indepencence-topology}
Consider the groupoid of fixed points $\mr X^G$, see Definition \ref{def:fix groupoid}, with attached topological space $\va{\mr X^G}$. Consider $\va{\ca X(\RR)}$ as a topological space via the real analytic topology. 
 The bijection $\va{F} \colon \va{\mr X^G} \cong \va{\ca X(\RR)}$ of Lemma \ref{lemma:realDMstackasgrupoid2} is a homeomorphism. 
\end{theorem}
\begin{proof}
By Proposition \ref{prop:independence}, to prove the theorem, we may replace our surjective \'etale morphism $U \to \ca X$ by any other surjective \'etale morphism $U' \to \ca X$ with $U'$ a scheme. In particular, by Theorem \ref{theorem:DGF-reallocusstack}, we may assume that $U(\RR) \to \va{\ca X(\RR)}$ is surjective. By definition of the real analytic topology on $\va{\ca X(\RR)}$, it therefore suffices to show that the topology of $\va{\mr X^G}$ agrees with the quotient topology induced by the surjection $U(\RR) \to \va{\mr X^G}$ and the real analytic topology of $U(\RR)$. This holds, because the map $U(\RR) \to \va{\mr X^G}$ is open by Lemmas \ref{lem:quotient-open} and \ref{lemma:st-open}. We are done. 
\end{proof}
\section{Topology of real gerbes}\label{sec:top-of-gerbes}
In this section, we prove Theorem \ref{thm:criterionetale:intro}. To that end, in Section \ref{sec:inertiaoverC}, we begin by studying the topology of the complex inertia $\va{\ca I_{\ca X}(\CC)}$ of such a gerbe $\ca X$. Section \ref{fibers} is devoted to basic properties of the map on complex points $\va{\ca I_{\ca X}(\CC)} \to \va{\ca X(\CC)}$ induced by the inertia morphism $\ca I_{\ca X} \to \ca X$, and of the maps on real points induced by the natural morphism from a separated Deligne–Mumford stack to its coarse moduli space. In Section \ref{sec:families}, we study the stack $[U/H]$ associated to a finite \'etale group scheme $H \to U$ over a real variety $U$. Finally, in Section \ref{sec:gerbes}, we combine these ingredients to prove Theorem \ref{thm:criterionetale:intro}.

\subsection{Preliminaries on the complex inertia.}\label{sec:inertiaoverC}
Let \( \mathcal{X} \) be a separated Deligne--Mumford stack locally of finite type over $\CC$, and let \( \mathcal{I}_{\mathcal{X}} \) denote its inertia stack. Let \( \mathcal{X} \to M \) be the coarse moduli space of \( \mathcal{X} \), and let \( \mathcal{I}_{\mathcal{X}} \to I_{\mathcal{X}} \) be the coarse moduli space of \( \mathcal{I}_{\mathcal{X}} \). 
A key property that we will use frequently in this paper and its sequel \cite{ambrosi2025topologyrealalgebraicstacks} is that the morphism of complex analytic spaces  $I_{\mathcal{X}}(\mathbb{C}) \to M(\mathbb{C})$ (induced by the natural map \( \mathcal{I}_{\mathcal{X}} \to \mathcal{X} \)) is closed and has finite fibers. While this result may be known to experts, we provide a proof in Lemma \ref{cor:coarse-proper} below due to the lack of a suitable reference. The proof relies on the following technical lemma.

\begin{lemma} \label{lemma:crucial-finiteness}
Let $\ca X$ and $\ca Y$ be separated Deligne--Mumford stacks locally of finite type over $\CC$, with coarse moduli spaces $\ca X \to M_{\ca X}$ and $\ca Y \to M_{\ca Y}$. For a finite morphism of stacks $\ca X \to \ca Y$, the induced morphism of coarse moduli spaces $M_{\ca X} \to M_{\ca Y}$ is finite. 
%Let $f \colon \ca X \to \ca Y$ be a morphism of separated Deligne--Mumford stacks locally of finite type over $\CC$, with coarse 
%Let $\ca X \to M$ and $\ca Y \to N$ be the respective coarse moduli spaces, and let $g \colon M \to N$ be the morphism induced by $f$. If $f$ is finite, then $g$ is finite. 
\end{lemma}
\begin{proof}
First, we claim that the morphism on coarse moduli spaces $M_{\ca X} \to M_{\ca Y}$ induced by $\ca X \to \ca Y$ is separated. Since $M_{\ca X}$ and $M_{\ca Y}$ are coarse moduli spaces of separated Deligne--Mumford stacks over $\CC$, they are both separated over $\CC$ (cf.\ \cite[Theorem 1.1]{conrad-stacks}). %Moreover, since $\ca Y$ is separated, $Z$ and $M$ are separated (for $M$, this follows from \cite[Theorem 1.1]{conrad-stacks}). 
Consider the factorization $M_{\ca X} \xrightarrow{\alpha} M_{\ca X} \times_{M_{\ca Y}} M_{\ca X} \xrightarrow{\beta} M_{\ca X} \times_\CC M_{\ca X}$ of the diagonal $\Delta = \beta \circ \alpha \colon M_{\ca X} \to M_{\ca X} \times_{\CC}M_{\ca X}$. The map $\beta$, which is the base change of $M_{\ca Y} \to M_{\ca Y} \times_{\CC} M_{\ca Y}$ along the map $M_{\ca X} \times_\CC M_{\ca X} \to M_{\ca Y} \times_\CC M_{\ca Y}$, is a closed immersion since $M_{\ca Y}$ is separated. 
As $M_{\ca X}$ is separated, $\Delta = \beta \circ \alpha$ is a closed immersion. Hence $\alpha$ is a closed immersion by \cite[\href{https://stacks.math.columbia.edu/tag/0AGC}{Tag 0AGC}]{stacks-project}, so that $M_{\ca X} \to M_{\ca Y}$ is separated, as wanted.

Next, pick a finite surjective morphism $Y \to \ca Y$ where $Y$ is a scheme; such a morphism exists by \cite[Theorem 16.6, page 153]{laumon-moretbailly}. We claim that the composition $Y \to \ca Y \to  M_{\ca Y}$ is finite. To see this, note that $Y \to M_{\ca Y}$ is proper and quasi-finite because $\ca Y \to M_{\ca Y}$ is proper and quasi-finite (see \cite[Theorem 1.1]{conrad-stacks}) and $Y \to \ca Y$ is finite. Therefore, $Y \to M_{\ca Y}$ is finite, see \cite[Corollaire (A.2.1), page 198]{laumon-moretbailly}. 

Define $X = Y \times_{\ca Y} \ca X$. Since $\ca X \to \ca Y$ is finite, the base change $X \to Y$ is finite. As $Y \to M_{\ca Y}$ is finite by the above, we conclude that the composition $X \to Y \to  M_{\ca Y}$ is finite. This composition agrees with the composition $X \to M_{\ca X} \to M_{\ca Y}$, in which $X \to M_{\ca X}$ is surjective and $X \to M_{\ca Y}$ is finite. Hence $M_{\ca X} \to M_{\ca Y}$ is locally quasi-finite by \cite[\href{https://stacks.math.columbia.edu/tag/03MJ}{Tag 03MJ}]{stacks-project} and \cite[\href{https://stacks.math.columbia.edu/tag/0GWS}{Tag 0GWS}]{stacks-project}. 
Since $X \to M_{\ca X}$ is surjective and $X \to M_{\ca Y}$ quasi-compact, the morphism $M_{\ca X} \to M_{\ca Y}$ is quasi-compact (see \cite[\href{https://stacks.math.columbia.edu/tag/040W}{Tag 040W}]{stacks-project}). We conclude that $M_{\ca X} \to M_{\ca Y}$ is quasi-finite, and in particular of finite type. 

We claim that $M_{\ca X} \to M_{\ca Y}$ is proper. This follows from \cite[\href{https://stacks.math.columbia.edu/tag/08AJ}{Tag 08AJ}]{stacks-project} because  $M_{\ca X} \to M_{\ca Y}$ is separated and of finite type, $X \to M_{\ca X}$ is surjective, and $X \to M_{\ca Y}$ is proper (being the composition of the proper morphisms $X \to Y$ and $Y \to M_{\ca Y}$). 

We conclude that $M_{\ca X} \to M_{\ca Y}$ proper and quasi-finite. By \cite[Corollaire (A.2.1), page 198]{laumon-moretbailly}, this implies that $M_{\ca X} \to M_{\ca Y}$ is finite.  
\end{proof}

%Lemma \ref{lemma:crucial-finiteness} allows us to prove:

\begin{lemma} \label{cor:coarse-proper}
Let $\ca X$ be a separated Deligne--Mumford stack locally of finite type over $\CC$. Let $\ca I_{\ca X} \to I_{\ca X}$ and $\ca X \to M$ be the respective coarse moduli spaces of $\ca I_{\ca X}$ and $\ca X$. Note that the natural morphism $\ca I_{\ca X} \to \ca X$ induces a morphism $I_{\ca X} \to M$. % and hence $I_{\ca X}(\CC) \to M(\CC)$. 
\begin{enumerate}
    \item \label{item:coarse:one} The morphism $\ca I_{\ca X}\to \ca X$ is finite. 
    \item \label{item:coarse:two}The morphism $I_{\ca X} \to M$ is finite and surjective. 
    \item \label{item:coarse:three}The induced morphism of complex analytic spaces $I_{\ca X}(\CC) \to M(\CC)$ is surjective and closed with finite fibers. 
\end{enumerate}
\end{lemma}
\begin{proof}
The diagonal morphism $\Delta \colon \ca X \to \ca X \times_\CC \ca X$ is of finite type, see \cite[Lemme (4.2), page 26]{laumon-moretbailly}. Therefore, for each scheme $S$ over $\CC$ and each $x \in \ca X(S)$, the automorphism group scheme $\ul{\Aut}_S(x)$ of $x$ over $S$ is of finite type over $S$. Moreover, $\Delta$ is quasi-finite by \cite[Lemme (4.2)]{laumon-moretbailly}. In particular, %$\Delta$ is finite. %Let $\ca X$ be a separated Deligne--Mumford stack of finite type over $\CC$. For 
%for each scheme $S$ over $\CC$ and each $x \in \ca X(S)$, the automorphism group scheme 
$\ul{\Aut}_S(x)$ is finite over $S$. This proves %that $\ca I_{\ca X} \to \ca X$ is finite, and hence 
item \ref{item:coarse:one}. As $I_{\ca X} \to M$ is surjective, item \ref{item:coarse:two} follows from item \ref{item:coarse:one} and Lemma \ref{lemma:crucial-finiteness}. Finally, if a map of analytic spaces $f_\CC \colon X(\CC) \to Y(\CC)$ is induced by a finite surjective morphism $f \colon X\to Y$ of algebraic spaces which are locally of finite type over $\CC$, then $f_\CC$ is surjective and closed with finite fibers. Item \ref{item:coarse:three} follows thus from item \ref{item:coarse:two}. 
%Hence, since \( I_{\mathcal{X}} \to M \) is finite and surjective by what has already been proven, we are done.
\end{proof}

For an algebraic stack $\ca X$, see \cite[\href{https://stacks.math.columbia.edu/tag/06QC}{Tag 06QC}]{stacks-project} for the notions of \emph{gerbe} and \emph{gerbe over an algebraic stack $\ca Y$}. For example, $\ca X$ is a gerbe if there exists an algebraic space $X$ and a morphism $\ca X \to X$ which turns $\ca X$ into a gerbe over $X$. See also \cite[(3.15) -- (3.21), pages 22 -- 24]{laumon-moretbailly}. % see also \cite[\href{https://stacks.math.columbia.edu/tag/06QB}{Tag 06QB}]{stacks-project}.
%so that $W \to M$ is finite surjective. 
%Hence $I_{\ca X} \to M$ is locally quasi-finite by \cite[\href{https://stacks.math.columbia.edu/tag/03MJ}{Tag 03MJ}]{stacks-project} and \cite[\href{https://stacks.math.columbia.edu/tag/0GWS}{Tag 0GWS}]{stacks-project}, and quasi-compact since $W \to M$ is quasi-compact. 

\begin{proposition}\label{lem:dm-flatness}
Let $\ca X$ be a separated Deligne--Mumford stack locally of finite type over $\CC$. Consider the following assertions.
\begin{enumerate}
    \item \label{item:ii}
The inertia $\pi\colon \ca I_{\ca X} \to \ca X$ is \'etale over $\ca X$. 
\item \label{item:iii}The inertia $\pi\colon \ca I_{\ca X} \to \ca X$ is flat over $\ca X$. 
\item \label{item:iv}The coarse moduli space map $\ca X \to M$ is a gerbe. 
\item \label{item:i}We have that $\# \Aut(x)$ is locally constant for $x \in \ca X(\CC)$. 
\end{enumerate}
Then $\ref{item:ii}\Leftrightarrow \ref{item:iii}\Leftrightarrow \ref{item:iv}\Rightarrow \ref{item:i}$. If $\ca X$ is reduced, then all four assertions are equivalent. 
\end{proposition}
\begin{proof}
We start by proving the equivalence of \ref{item:iii} and \ref{item:iv}.
Note that $\pi \colon \ca I_{\ca X} \to \ca X$ is finite by Lemma \ref{cor:coarse-proper}. Therefore, 
by \cite[\href{https://stacks.math.columbia.edu/tag/06QJ}{Tag 06QJ}]{stacks-project}, the morphism $\pi \colon \ca I_{\ca X} \to \ca X$ is flat if and only if $\ca X$ is a gerbe. Hence, we need to prove that $\ca X$ is a gerbe if and only if $\ca X$ is a gerbe over $M$. By definition, if $\ca X$ is a gerbe over $M$ then $\ca X$ is a gerbe. Conversely, assume that $\ca X$ is a gerbe. Let $\pi_0(\ca X)^{sh}$ be the sheafification of the presheaf $U \mapsto \rm{Ob}(\ca X(U))/_{\cong}$ on the site $(\Sch/\CC)_{fppf}$. By \cite[\href{https://stacks.math.columbia.edu/tag/06QD}{Tag 06QD}]{stacks-project}, we have that $\pi_0(\ca X)^{sh}$ is an algebraic space. On the one hand, by \cite[Lemme 3.18, page 23]{laumon-moretbailly}, the morphism $\ca X \to M$ factors uniquely as $\ca X \to \pi_0(\ca X)^{sh} \to M$. 
    On the other hand, since $\ca X \to M$ is the coarse moduli space of $\ca X$ and since $\pi_0(\ca X)^{sh}$ is an algebraic space, the morphism $\ca X \to \pi_0(\ca X)^{sh}$ factors uniquely as $\ca X \to M \to \pi_0(\ca X)^{sh}$. This proves that $M \cong \pi_0(\ca X)^{sh}$. Finally, by \cite[\href{https://stacks.math.columbia.edu/tag/06QD}{Tag 06QD}]{stacks-project}, the natural morphism $\ca X \to \pi_0(\ca X)^{sh}$ turns $\ca X$ into a gerbe over $\pi_0(\ca X)^{sh} \cong M$. This proves the equivalence of assertions \ref{item:iii} and \ref{item:iv}. 

Trivially, assertion \ref{item:ii} implies assertion \ref{item:iii}. Moreover, as a finite flat group scheme of order invertible in the base is finite étale, flatness of $\pi$ implies \'etaleness of $\pi$. Thus, assertions \ref{item:ii} and \ref{item:iii} are equivalent, and assertion \ref{item:ii} readily implies assertion \ref{item:i}.

To conclude the proof, it is enough to show that assertion \ref{item:i} implies assertion \ref{item:iii} if $\ca X$ is reduced. So, assume that $\ca X$ is reduced and that $\# \Aut(x)$ is locally constant for $x \in \ca X(\CC)$.  Since $\ca X$ is reduced, there exists a reduced scheme $U$ and a surjective \'etale morphism $U \to \ca X$. Let $x \in \ca X(U)$ be the object corresponding to $U \to \ca X$. We have ${\ca I_{\ca X}} \times_{\ca X} U = \underline{\Aut}_U(x) \to U$, the automorphism group scheme of $x$ over $U$, and by \cite[\href{https://stacks.math.columbia.edu/tag/04XD}{Tag 04XD}]{stacks-project}, it suffices to show that $\underline{\Aut}_U(x) \to U$ is flat. By Lemma \ref{cor:coarse-proper}, the morphism $\underline{\rm{Aut}}_U(x) \to U$ is finite, and by Cartier's theorem it has reduced fibres. As $\#\Aut(x)$ is locally constant for $x \in \ca X(\CC)$, it follows that the function $\rho \colon U \to \ZZ$ defined as $\rho(u) = \dim_{k(u)}(f_U^{-1}(y))$ is locally constant; since $U$ is reduced, $\underline{\rm{Aut}}_U(x) \to U$ is therefore flat by \cite[\href{https://stacks.math.columbia.edu/tag/00NX}{Tag 00NX}]{stacks-project} and \cite[\href{https://stacks.math.columbia.edu/tag/0FWG}{Tag 0FWG}]{stacks-project}. 
\end{proof}
\subsection{The fibers of the maps $\va{\mathcal I_{\mathcal X}(\CC)}\rightarrow \va{\mathcal X(\CC)}$ and $\va{\mathcal X(\mathbb R)}\rightarrow M(\mathbb R)$.}\label{fibers}

For later use, and to provide further motivation for Conjecture \ref{conj:ST-top}, we show that, in the case of a stacky quotient $\ca X = [X/\Gamma]$ of a real scheme $X$ by a finite group $\Gamma$, the fibers of the maps $\pi \colon \va{\ca I_{\mathcal X}(\mathbb 
 C)} \rightarrow \vert\ca X(\CC)\vert = (X/\Gamma)(\CC)$ and $f \colon \vert \mathcal X(\mathbb R)\vert \rightarrow (X/\Gamma)(\mathbb R)$ are closely related. For example, by combining Propositions \ref{prop:fibres-of-inertia} and \ref{prop:fibre-H1}, one sees that for a point \( x \in X(\mathbb{R}) \) with image \( y \in (X/\Gamma)(\mathbb{R}) \), we have  
$
\pi^{-1}(y) = \Gamma_x / \Gamma_x$ and $f^{-1}(y) = \mathrm{H}^1(G, \Gamma_x)
$. Here, $\Gamma_x/\Gamma_x$ denotes the set of conjugacy classes of $\Gamma_x$. 

\subsubsection{Fibres of the coarse inertia map.} Let $S$ be a scheme of finite type over $\mathbb C$. Let $X$ be a scheme and $f\colon X \to S$ a scheme of finite type over $S$. Let $H \to S$ be a finite group scheme over $S$, equipped with an action on $X$ over $S$. 

Define $\ca X$ as the quotient stack $[X/H]$, which we view as a stack over $\CC$ equipped with a natural morphism $\ca X \to S$.  Let $q \colon X(\CC) \to \va{\ca X(\CC)}$ be the natural map of spaces over $S(\CC)$, giving the quotient map $q_s\colon X_s(\CC) \to \va{\ca X(\CC)}_s = X_s(\CC)/H_s(\CC)$ for $s \in S(\CC)$. 
Consider the canonical map $\va{\pi} \colon \va{\ca I_{\ca X}(\CC)} \to \va{\ca X(\CC)}$ of topological spaces over $S(\CC)$. 

\begin{proposition}  \label{prop:fibres-of-inertia}
In the above notation, the following holds. 
\begin{enumerate}
\item There is a canonical bijection \begin{align} \label{align:inertia-description}
\va{\ca I_{\ca X}(\CC)}= \set{\left(x \in X(\CC), \gamma \in \Stab_{H_{f(x)}(\CC)}(x) \right)} /_{\sim}
\end{align}
where $(x,\gamma) \sim (gx, g\gamma g^{-1})$ for $g \in H_{f(x)}(\CC)$. 
\item Fix $s \in S(\CC)$, and consider the induced map $\va{\pi}_s \colon \va{\ca I_{\ca X}(\CC)}_s \to \va{\ca X(\CC)}_s$. Fix $x \in \va{\ca X(\CC)}_s$. 
There is a canonical bijection
\begin{align}\label{align:item222}
\va{\pi}^{-1}_s(x) = \left( \coprod_{y \in q_s^{-1}(x)} \Gamma_{y} \right)/\Gamma \quad \quad \left(\Gamma = H_s(\CC), \quad \Gamma_y = \Stab_\Gamma(y)\right).
\end{align}
Here, $g \in \Gamma$ acts on $\bigsqcup_{y \in q_s^{-1}(x)} \Gamma_{y}$ as follows: for $y \in q^{-1}(x)$, $\gamma \in \Gamma_y$, we define 
$
g \cdot (y, \gamma) = (gy, g \gamma g^{-1})$. In particular, for fixed $y'\in q^{-1}_s(x)$, there are bijections
\begin{align} \label{align:bij-stab-non-can}
\va{\pi}^{-1}_s(x) = \left( \coprod_{y \in q_s^{-1}(x)} \Gamma_{y} \right)/\Gamma 
\cong \Gamma_{y'} /\Gamma_{y'},
\end{align}
of which the second one is in general non-canonical. 
\end{enumerate}

\end{proposition}
\begin{proof}
Let $\rm{Stab}_S \to X$ be the stabilizer group scheme attached to the action of $H$ on $X$ over $S$. Then
$
\rm{Stab}_S(\CC) = \set{(x, \gamma) \in X(\CC) \times_{S(\CC)} H(\CC) \mid \gamma x = x}$. The group scheme $H$ acts on the scheme $\rm{Stab}_S$ over $S$ by 
$
g \cdot (x, \gamma) = (gx, g\gamma g^{-1})$ for $(g, (x, \gamma)) \in H \times_S \rm{Stab}_S$. 
We have a canonical isomorphism of stacks 
$
\ca I_{\ca X} = [\rm{Stab}_S/H]$ (see e.g.\ \cite[Exercise 3.2.12]{notesstack}). In particular, \eqref{align:inertia-description} follows. This proves item 1. Then \eqref{align:item222} follows from item 1. It remains to provide the second bijection in \eqref{align:bij-stab-non-can}. This holds, since for each $y_1, y_2 \in q^{-1}(x)$, there exists $g \in \Gamma$ such that $g y_1 = y_2$ and $g\Gamma_{y_1} g^{-1} = \Gamma_{y_2}$. 
\end{proof}

%\begin{remark}
 %   \begin{enumerate}
%   \item 
%In the notation of Proposition \ref{prop:fibres-of-inertia}, assume that $S = \Spec(\CC)$, and that $\Gamma \coloneqq H(\CC)$ acts freely on $X$. Then $\Gamma_y = \set{e}$ for each $y \in q^{-1}(x)$, and $\Gamma$ acts freely on $q^{-1}(x)$. Hence 
 %   $\va{\pi}^{-1}(x)$ is a singleton.
%In the notation of Proposition \ref{prop:fibres-of-inertia}, assume that $S = \Spec(\CC)$, and that $\Gamma \coloneqq H(\CC)$ is abelian. Then the bijection in \eqref{align:bij-stab-non-can} is canonical: there is a canonical bijection between $(\bigsqcup_{y \in q^{-1}(x)} \Gamma_{y})/\Gamma$ and $\Gamma_y = \Gamma_{y'} \subset \Gamma$ for any $y,y' \in q^{-1}(x)$. 
%\end{enumerate}
%\end{remark}

\subsubsection{Fibres of the map to the real locus of the coarse moduli space.} The next proposition allows one to understand the fibres of the map on real loci $\va{\ca X(\RR)} \to M(\RR)$ induced by the coarse moduli space map $\ca X \to M$ of a real DM stack $\ca X$. See also \cite{HS} for related results in the context of quiver varieties
\begin{proposition} \label{prop:fibre-H1}
    Let $\ca X$ be a real DM stack, with coarse moduli space $p \colon \ca X \to M$. Let $f \colon \vXR \to M(\RR)$ denote the map induced by $p$, and let $x \in \ca X(\RR)$ with isomorphism class $[x] \in \va{\ca X(\RR)}$ (cf.\ Section \ref{sec:notation}).
    \begin{enumerate}
        \item There is a canonical bijection $f^{-1}(f([x])) = \rm  H^1(G, \Aut(x_\CC))$.  
        \item We have $\# \rm  H^1(G, \Aut(x_\CC)) = \# \rm  H^1(G, \Aut(x'_\CC))$ for each pair of objects $x, x' \in \ca X(\RR)$ whose induced objects $x_\CC, x'_\CC \in \ca X(\CC)$ are isomorphic in $\ca X(\CC)$. 
        %with $f([x]) = f([x'])$. Then
    \end{enumerate}
\end{proposition}
\begin{proof}
Since two objects in $\ca X(\CC)$ are isomorphic if and only if their images in $M(\CC)$ are the same, the second item is a consequence of the first item. The first item follows from \cite[Section 4]{grothendieck-descente-I}. 
\end{proof}

\subsection{Topological groupoids and families of finite $G$-groups.}\label{sec:families} Let $\pi \colon H \to U$ be a locally trivial family of finite topological $G$-groups. This means that $\pi$ is a finite topological covering, that there are involutions $\sigma \colon H \to H, \sigma \colon U \to U$ commuting with $\pi$,
    and that there is a continuous group law $m \colon H \times_U H \to H$, an inversion $i \colon H \to H$ and identity $e \colon U \to H$ all compatible with the involutions $\sigma$; moreover, we require that for each $x \in U$ there exists an open neighbourhood $x \in V \subset U$ such that $H|_V \cong V \times \Gamma$ as families of topological groups, for a finite group $\Gamma$. 
\begin{definition} \label{def:H1GH}
Let $\pi \colon H \to U$ be a locally trivial family of finite topological $G$-groups. We define 
\begin{align*}
\rm Z^1(G,H) &\coloneqq \set{(u, g) \in U \times H \mid u \in U^G, g \in H_u \text{ and } g \sigma(g) = e}, \\
\rm  H^1(G,H) &\coloneqq \rm Z^1(G,H) /\sim
\end{align*}
where $(u, g) \sim (u',g')$ if $u = u'$ and there exists $h \in H_u$ such that $g' = h g \sigma(h)^{-1}$. %There is a natural surjective map $\rm  H^1(G,H) \to U^G$, and for $V \subset U^G$ open such that $H|_V \cong V \times \Gamma$ as families of topological $G$-groups, for a finite $G$-group $\Gamma$, we have $\rm  H^1(G,H)|_V = \rm  H^1(G,H|_V) \cong \rm  H^1(G,\Gamma) \times V$. This proves a topology on the subsets $\rm  H^1(G,H)|_V$ for $V \subset U^G$ open, which glue to define a topology on $\rm  H^1(G,H)$. 
We equip $ \rm Z^1(G,H)$ with the subspace topology coming from $U \times H$ and we equip $\rm  H^1(G,H)$ with the quotient topology coming from $\rm Z^1(G,H)$. 
\end{definition}

\begin{remark} \label{rem:HUgroupoid}
    Note that $\pi \colon H \to U$ defines a topological groupoid $\mr X = [H \rightrightarrows U]$ enhanced with a natural involution $\sigma \colon \mr X \to \mr X$. The maps $s,t \colon H \to U$ are the same and both equal $\pi$, and the composition map $c\colon H \times_U H \to H$ equals the group law morphism $m$.   
Moreover, by construction, the space $\rm H^1(G,H)$ is nothing but the coarse space $\va{\mr X^G}$ of the groupoid of fixed points $\mr X^G$ constructed in Definition \ref{def:fix groupoid}. 
\end{remark}

\begin{lemma} \label{lem:topcovH1}
Let $\pi \colon H \to U$ be a locally trivial family of finite topological $G$-groups. 
    There is a natural surjective map $\rm  H^1(G,H) \to U^G$, and for $V \subset U^G$ open such that $H|_V \cong V \times \Gamma$ as families of topological $G$-groups over $V$ for some finite $G$-group $\Gamma$, we have $\rm  H^1(G,H)|_V = \rm  H^1(G,H|_V) \cong \rm  H^1(G,\Gamma) \times V$. In particular, the natural map 
    \[
    \rm  H^1(G,H) \longrightarrow U^G
    \]
    is a topological covering, with fibre $\rm  H^1(G,H_u)$ for $u \in U^G$. 
\end{lemma}
\begin{proof}
This follows from the fact that the construction of $\rm H^1(G,H)$ commutes with base change along a map $V \to U$ of topological $G$-spaces. 
\end{proof}

\begin{proposition} \label{proposition:comparison-alg-top}
Let $H \to U$ be a finite \'etale group scheme over a scheme $U$ of finite type over $\RR$. Consider the associated quotient stack $[U/H]$, and also the associated locally trivial family of finite $G$-groups $H(\CC)\to U(\CC)$. 
There is a canonical homeomorphism
\[
\va{[U/H](\RR)} \xrightarrow{\sim} \rm  H^1(G, H(\CC))
\]
of spaces over $U(\RR)$, where the space on the right is defined in Definition \ref{def:H1GH}. 
\end{proposition}
\begin{proof}
As the group scheme $H \to U$ is finite \'etale, the morphism $H(\mathbb C)\rightarrow U(\mathbb C)$ is a locally trivial family of finite topological $G$-groups. Hence, in view of Remark \ref{rem:HUgroupoid}, the proposition is a special case of Theorem \ref{theorem:indepencence-topology}. 
%The set $\va{[U/H](\RR)}$ parametrizes pairs $(u, P)$ where $u \colon \Spec(\RR) \to U$ is an $\RR$-point and $P$ is a $H_u$-torsor over $\RR$. Such a pair $(u,P)$ corresponds to an element $\gamma(u,P) \in \rm  H^1(G,H_u(\CC)) \subset \rm  H^1(G, H)$ (see e.g. \cite[Example 2, Page 13]{MR1845760}).
%This gives the bijection fibrewise over $U(\RR)$; this bijection is a homeomorphism by Theorem \ref{theorem:indepencence-topology} and Remark \ref{rem:HUgroupoid}. 
%proving the lemma.
\end{proof}
\subsection{Gerbes and topological coverings on real loci.}\label{sec:gerbes}
In this section we prove Theorem \ref{thm:criterionetale:intro}. Before we can start with the proof, we need three preliminary results. 
 \begin{lemma} \label{lem:et-local-hom}
     Let $f \colon X \to Y$ be a morphism of schemes $X,Y$ which are locally of finite type over $\RR$. If $f$ is \'etale, then $f_\RR \colon X(\RR) \to Y(\RR)$ is a local homeomorphism. 
 \end{lemma}
 \begin{proof}
     Consider the map of complex analytic spaces $f_\CC \colon X(\CC) \to Y(\CC)$. This map is a local homeomorphism by \cite[Exposé XII, Proposition 3.1 \& Remarque 3.3]{SGA1}. Hence the same holds for the restriction to real points and the lemma follows from this. 
 \end{proof}
 \begin{lemma} \label{lem:top-cov-basechange}
    Let $f \colon X \to Y$ be a map of topological spaces, let $\pi \colon Y' \to Y$ be a local homeomorphism with $\Ima(\pi) = \Ima(f)$. Assume that the base change $f' \colon X' \coloneqq X \times_Y Y' \to Y'$ is open and a topological covering over its image. Then $f$ is open and a topological covering over its image.
 \end{lemma}
 \begin{proof}
 This holds, since, for a morphism of topological spaces, the property of being an open map and a topological covering over its image is local on the target. 
 \end{proof}
For an algebraic stack $\ca X$, let $\ca X_{\rm{red}} \subset \ca X$ denote its reduction (cf.\  \cite[\href{https://stacks.math.columbia.edu/tag/050C}{Tag 050C}]{stacks-project}). 

\begin{lemma}\label{lemma:reduction}
    Let $\ca X$ be a separated Deligne--Mumford stack of finite type over $\RR$. Let $\ca X_{\rm{red}} \subset \ca X$ be the reduction of $\ca X$. Let $\ca X \to M$ and $\ca X_{\rm{red}} \to N$ be the coarse moduli spaces. The natural maps $\va{\ca X_{\rm{red}}(\RR)} \to \va{\ca X(\RR)}$ and $N(\RR) \to M(\RR)$ are homeomorphisms. 
    %The coarse moduli space $\ca X_{\rm{red}}  \to M_{\rm{red}}$ of $\ca X_{\rm{red}}$ is the reduction $M_{\rm{red}} \subset M$ of $M$. 
\end{lemma}
\begin{proof}
Recall first that $\ca X_{\rm{red}} \to \ca X$ is a closed immersion, hence fully faithful 
(cf.\ \cite[\href{https://stacks.math.columbia.edu/tag/0504}{Tag 0504}]{stacks-project}, \cite[\href{https://stacks.math.columbia.edu/tag/04ZZ}{Tag 04ZZ}]{stacks-project}). In particular, the map $\va{\ca X_{\rm{red}}(\RR)} \to \va{\ca X(\RR)}$ is injective. Let $U$ be a scheme and let $U \to \ca X$ be a surjective \'etale morphism such that $U(\RR) \to \va{\ca X(\RR)}$ is surjective (see Theorem \ref{theorem:DGF-reallocusstack}). Let $V = \ca X_{\rm{red}} \times_{\ca X} U$. Then $V \to \ca X_{\rm{red}}$ is \'etale and essentially surjective on real points. In particular, since $\ca X_{\rm{red}}$ is reduced, $V$ is reduced. The map $V \to U$ is a closed immersion. Thus $V = U_{\rm{red}}$, the reduction of $U$, so that $V(\RR) = U(\RR)$. It follows that $\va{\ca X_{\rm{red}}(\RR)} \to \va{\ca X(\RR)}$ is surjective, hence bijective. Moreover, since $V = U_{\rm{red}}$, the bijection $\va{\ca X_{\rm{red}}(\RR)} \cong \va{\ca X(\RR)}$ is a homeomorphism (cf.\ Definition \ref{def:real-an-top}). 

Next, we claim that the map $N \to M$ is a universal homeomorphism.  
The map $\ca X_{\rm{red}} \to \ca X$ is a universal homeomorphism \cite[\href{https://stacks.math.columbia.edu/tag/054M}{Tag 054M}]{stacks-project} and the map $\ca X \to M$ is a universal homeomorphism \cite[\href{https://stacks.math.columbia.edu/tag/0DUT}{Tag 0DUT}]{stacks-project}. Thus, the composition $\ca X_{\rm{red}} \to N \to M$, which agrees with the composition $\ca X_{\rm{red}} \to \ca X \to M$, is also a universal homeomorphism, and hence $N \to M$ is a universal homeomorphism (see the proof of \cite[\href{https://stacks.math.columbia.edu/tag/0H2M}{Tag 0H2M}]{stacks-project}). 

It follows that $N \to M$ is universally injective. In particular, the map $N(\RR) \to M(\RR)$ is injective \cite[\href{https://stacks.math.columbia.edu/tag/040X}{Tag 040X}]{stacks-project} and the morphism $N \to M$ is radicial \cite[\href{https://stacks.math.columbia.edu/tag/0484}{Tag 0484}]{stacks-project}. Since $N \to M$ is radicial and surjective, the map $N(\RR) \to M(\RR)$ is also surjective (see \cite[\href{https://stacks.math.columbia.edu/tag/0481}{Tag 0481}]{stacks-project}). We conclude that the map $N(\RR) \to M(\RR)$ is a continuous bijection. Furthermore, the map $\ca X_{\rm{red}} \to N$ is surjective, the map $\ca X_{\rm{red}} \to M$ is proper, and $M$ is separated (see \cite[Theorem 1.1]{conrad-stacks}). Hence, the map $N \to M$ is proper (see \cite[\href{https://stacks.math.columbia.edu/tag/0CQK}{Tag 0CQK}]{stacks-project}). Consequently, the continuous bijection $N(\RR) \cong M(\RR)$ is closed, and hence a homeomorphism. 
\end{proof}
 
\begin{proof}[Proof of Theorem \ref{thm:criterionetale:intro}] By Proposition \ref{lem:dm-flatness}, we know that $\ca X \to M$ is a gerbe, and that $\ca I_{\ca X} \to \ca X$ is finite \'etale. 
The proof proceeds in two steps. 

\textbf{Step 1}: \emph{To prove the theorem, we may assume that $\ca X$ is reduced and that the coarse moduli space map $\ca X \to M$ has a section.} To prove this, note that by Lemma \ref{lemma:reduction}, we may assume that $\ca X$ is reduced. Then $\ca X \to M$ is a gerbe, see Proposition \ref{lem:dm-flatness}. Let $U \to \ca X$ be a surjective \'etale morphism where $U$ is a scheme over $\RR$, such that $U(\RR) \to \va{\ca X(\RR)}$ is surjective. We look at the base change $\ca Y \coloneqq \ca X \times_M U$. % which fits in $2$-cartesian diagrams
%\[
%\xymatrix{
%\ca Y \ar[r] \ar[d] & U \ar[d] \\
%\ca X \ar[r] & M.
%}
%\]
The morphism $\ca X \to M$ is \'etale, as it is fppf locally on $M$ of the form $[U/H] \to U$ for a finite group scheme $H \to U$ (see \cite[\href{https://stacks.math.columbia.edu/tag/06QH}{Tag 06QH}]{stacks-project}) and $H \to U$ is \'etale since $\ca I_{\ca X} \to \ca X$ is \'etale (see Proposition \ref{lem:dm-flatness}). 
Hence, the composition $U \to \ca X \to M$ is \'etale. Therefore, by Lemma \ref{lem:et-local-hom}, the map
$
U(\RR) \to M(\RR)
$
is a local homeomorphism, whose image is the image of $\va{\ca X(\RR)} \to M(\RR)$. %Note also that the images of $U(\RR) \to M(\RR)$ and $\va{\ca X(\RR)} \to M(\RR)$ agree. 
Since $\ca X \times_M \ca X \to \ca X$ has a section and $\ca Y \to U$ is the base change of $\ca X \times_M \ca X \to \ca X$ along $U \to \ca X$, the morphism $\ca Y \to U$, which is a gerbe by \cite[\href{https://stacks.math.columbia.edu/tag/06QE}{Tag 06QE}]{stacks-project}, has a section as well. By assumption, the map $\va{Y(\RR)} \to U(\RR)$ is therefore open and a topological covering over its image. As the map $\va{\ca Y(\RR)} \to \va{\ca X(\RR)} \times_{M(\RR)} U(\RR)$ is a homeomorphism, we conclude from Lemma \ref{lem:top-cov-basechange} that the map $\va{\ca X(\RR)} \to M(\RR)$ is open and a topological covering over its image. Step 2 follows. 
%the same holds for $\va{\ca X(\RR)} \to M(\RR)$. 
%We proceed in steps. 
%By \cite[Corollary 2.4]{thesis-degaayfortman}, there exists a surjective \'etale map $V\rightarrow \mathcal Y$ from a finite type $\mathbb R$-scheme $V$ such that $V(\mathbb R)\rightarrow \vert \mathcal Y(\mathbb R)\vert$ is surjective. 

%\emph{\textbf{Step 1}: The morphism $\ca Y \to M$ is a gerbe and $\ca I_{\ca Y} \to \ca Y$ is \'etale.} 

\textbf{Step 2:} \emph{The map $\va{\ca X(\RR)} \to M(\RR)$ is a topological covering when $\ca X$ is reduced and the coarse moduli space map $\ca X \to M$ has a section.} Indeed, by Proposition \ref{lem:dm-flatness}, the map $\ca X \to M$ is a gerbe. Thus, assuming that $\ca X \to M$ has a section, we have $\ca X = [U/H]$ for a scheme $U$ and a group scheme $H \to U$, see 
\cite[\href{https://stacks.math.columbia.edu/tag/06QG}{Tag 06QG}]{stacks-project}. Since $\ca I_{\ca X} \to \ca X$ is finite \'etale (see Proposition \ref{lem:dm-flatness}), the map $H \to U$ is finite \'etale. We have a homeomorphism $\va{[U/H](\RR)} \cong \rm  H^1(G, H(\CC))$ as spaces over $U(\RR)$ by Proposition \ref{proposition:comparison-alg-top}, and $\rm  H^1(G, H(\CC)) \to U(\RR)$ is a topological covering by Lemma \ref{lem:topcovH1}. 
\end{proof}
\section{Smith--Thom for classifying stacks and groupoid cohomology}\label{sec:classifying}
\subsection{Smith--Thom for classifying stacks}
\begin{proof}[Proof of Proposition \ref{propintro:STsclaffyingstack}]
Let $\mr X\coloneqq [\Gamma \rightrightarrows \rm{pt}]$. 
By Example \ref{ex:fixed-BGamma}, we have
$\va{\mr X^G}\simeq \rm H^1(G,\Gamma)$, where $\rm  H^1(G, \Gamma)$ is the finite discrete space $\rm Z^1(G,\Gamma)/\sim $ with $\rm Z^1(G,\Gamma)\subset \Gamma$ the set of $\gamma\in \Gamma$ such that $\gamma\sigma(\gamma)=e$ and $\sim$ the equivalence relation $\gamma \sim \beta \gamma \sigma(\beta)^{-1}$ for $\beta \in \Gamma$. Therefore, we have
$\dim \rm H^\ast(\va{\mr X^G},\ZZ/2)=\# \rm  H^1(G,\Gamma).$ Moreover, by Proposition \ref{prop:fibres-of-inertia}, we have $\va{\cal {I}_{\mr X}} \simeq \Gamma/\Gamma$ so that $\dim \rm H^\ast(\va{\cal {I}_{\mr X}}, \ZZ/2)=\# (\Gamma/\Gamma),$ where $\Gamma/\Gamma$ is set of conjugacy classes of $\Gamma$. 
So Proposition \ref{propintro:STsclaffyingstack} follows from the following group theoretic lemma, whose proof has been suggested to us by Will Sawin.
\end{proof}
\begin{lemma}\label{lem:sawin}
	Let $\Gamma$ be a finite group with an action of $G$.
	Then $\#\rm  H^1(G,\Gamma) \leq \# (\Gamma/\Gamma).$  %following inequality holds:
%	$$\#\rm  H^1(G,\Gamma) \leq \# (\Gamma/\Gamma).$$
\end{lemma}
\begin{proof}
Let $\sigma \colon \Gamma \to \Gamma$ be the involution corresponding to the $G$-action. 
Let $\sigma\text{-conj}$ be the equivalence relation on $\Gamma$  induced by the action of $\Gamma$ on its self by $\sigma$-conjugacy (i.e. $h$ acts by $h(g)=hg\sigma(h^{-1})$. For every $h\in \Gamma$, we let $\Stab_{\sigma}(h)$ (resp.\ $[h]_{\sigma}$) be the stabilizer (resp.\ the orbit) of $h$ for the $\sigma$-conjugacy action and $\Stab(h)$ (resp.\ $[h]$), the stabilizer (resp.\ the orbit) for the conjugacy action.

We claim the following chain of inequalities and equalities:
$$ \#\rm  H^1(G,\Gamma) \leq \#(\Gamma/\sigma\text{-conj})= \#(\Gamma/\Gamma)^{G}\leq \#(\Gamma/\Gamma).$$
Since the first and the last inequalities follow from the inclusions $\rm  H^1(G,\Gamma)\subset (\Gamma/\sigma\text{-conj})$ and $(\Gamma/\Gamma)^{G}\subset \Gamma/\Gamma$, we just need to prove the middle equality.

By the orbit-stabilizer theorem, we have
$
\sum_{g \in [g]_\sigma} \# \rm{Stab}_\sigma(g) = \#\Gamma$. Therefore,
\[
\sum_{g \in \Gamma} \# \Stab_\sigma(g) = \sum_{[g]_\sigma \in \Gamma/\sigma\text{-}\rm{conj}}\sum_{g \in [g]}\#\Stab_\sigma(g) = \#\Gamma \cdot \# ( \Gamma/\sigma\text{-}\rm{conj}). 
\]
Moreover,
\begin{align*}
\sum_{g \in \Gamma} \# \Stab_\sigma(g) &= \sum_{g \in \Gamma} \#\set{h \in \Gamma \mid hg\sigma(h)^{-1} = g}  = \sum_{h \in \Gamma} \#\set{g \in \Gamma \mid \sigma(h) = g^{-1}hg} \\
& = \sum_{[h] \in (\Gamma/\Gamma)^G} \sum_{h \in [h]}  \# \set{g \in \Gamma \mid \sigma(h) = g^{-1}hg} = \sum_{[h] \in (\Gamma/\Gamma)^G} \sum_{h \in [h]}  \# \Stab(h) \\
& = \sum_{[h] \in (\Gamma/\Gamma)^G}  \# \Gamma = \# (\Gamma/\Gamma)^G \cdot \# \Gamma,
\end{align*}
where the penultimate equality follows again from the orbit-stabilizer theorem (now applied to the action of $\Gamma$ on $\Gamma$ by conjugation). 
This concludes the proof.
\end{proof}

\subsection{Smith--Thom inequality for groupoid cohomology} \label{sec:smith-thom-orbifold}

 In the previous sections, we examined the topological spaces $\va{\mr X^G}$ and $\va{\mr X}$ associated to a topological $G$-groupoid $\mr X = [R \rightrightarrows U]$, and proposed Conjecture~\ref{conj:ST-top}.  
 
 In a complementary direction, it is natural to compare the \emph{groupoid cohomology} of $\mr X^G$ with the groupoid cohomology of $\mr X$, as another route to generalizing the Smith--Thom inequality. More precisely, let us suppose that $\mr X$ is an \'{e}tale groupoid---that is, the source and target maps $s, t \colon R \to U$ are local homeomorphisms---equipped with an involution $\sigma \colon \mr X \to \mr X$. Then, as shown in Lemma~\ref{lemma:st-open}, the fixed-point groupoid $\mr X^G$ is also \'{e}tale. In this context, one may consider the groupoid cohomology $\rm H^\ast_{\rm{grp}}$ of both $\mr X$ and $\mr X^G$ (cf.~\cite[Section 2.1]{crainic-moerdijk}).

Remark that the cohomology group $\rm H^\ast_{\rm{grp}}(\mr X, \ZZ/2)$ is not finite dimensional in general, hence it does not make sense to compare the dimension of $\rm H^\ast_{\rm{grp}}(\mr X^G, \ZZ/2)$ with the dimension of $ \rm H^\ast_{\rm{grp}}(\mr X, \ZZ/2)$. Instead, one can ask the following. 

\begin{question}\label{que:groupoid-question}
Let $(\mr X = [R \rightrightarrows U], \sigma \colon \mr X \to \mr X)$ be an \'{e}tale $G$-groupoid.
\begin{enumerate}
\item \label{q1} Is there a constant $C > 0$ such that 
\[
\frac{\dim \rm H^{\leq i}_{\rm{grp}}(\mr X^G,\ZZ/2)}{\dim \rm H^{\leq i}_{\rm{grp}}(\mr X,\ZZ/2)} \leq C \quad \quad \text{for each $i \geq 0$?}
\]

\item \label{q2} If such a bound exists, can it be made independent of $\mr X$ (and, in particular, of the involution $\sigma$)?
\end{enumerate}
\end{question}

If $\mr X$ is a topological orbifold arising as the quotient of a topological space $X$ by the action of a finite group $\Gamma$, then by \cite[Section 1.3]{MR1816220}, we have an isomorphism
$
\rm H^i_{\rm{grp}}(\mr X ,\ZZ/2) \simeq \rm H^i_\Gamma(X ,\ZZ/2),
$
where the right-hand side denotes the $\Gamma$-equivariant cohomology of $X$. Moreover, by Example~\ref{ex:fixed-BGamma}, if $\Gamma$ is a finite group with an involution $\sigma \colon \Gamma \to \Gamma$, and $\mr X = B\Gamma = [\Gamma \rightrightarrows \mathrm{pt}]$ is the classifying groupoid of $\Gamma$, then there is a natural isomorphism of topological groupoids
$
\mr X^G \cong \coprod_{[\gamma]} B(\Gamma^{\sigma_\gamma})$, where $[\gamma]$ ranges over the elements in $\rm H^1(G, \Gamma)$, where  $B(\Gamma^{\sigma_\gamma}) = [\Gamma^{\sigma_\gamma} \rightrightarrows \rm{pt}]$, and where, for an element $\gamma \in \Gamma$ with $\gamma \sigma(\gamma) = e$, the subgroup $\Gamma^{\sigma_\gamma}\subset \Gamma$ consists of those $g\in \Gamma$ such that $\gamma\sigma(g)=g\gamma$. 
\begin{proposition} \label{prop:yes-answer}
    Question~\ref{que:groupoid-question}.\ref{q1} has a positive answer when $\mr X = B\Gamma$ for a finite $G$-group $\Gamma$. 
\end{proposition}
\begin{proof}
By \cite[Corollary 2.2]{quillen}, the cohomology ring $\rm H^\ast(\Gamma, \ZZ/2)$ is a finitely generated graded $\bb F_2$-algebra. Consider the associated Poincar\'e series defined as
\[
P_\Gamma(t) \coloneqq \sum_{i = 0}^\infty \dim_{\FF_2} \rm H^i(\Gamma, \ZZ/2) \cdot t^i \in \ZZ[[t]].
\]
By a result of Venkov (see \cite[Proposition 2.5]{quillen}), one has that $P_H(t)$ is a rational function for any finite group $H$. For each $\gamma \in \Gamma$ such that $\gamma\sigma(\gamma) = e$, we obtain a rational function
\begin{align}\label{align:rational-function-quillen}
P_{\Gamma, \gamma}(t) \coloneqq \frac{P_{\Gamma^{\sigma_\gamma}}(t)}{P_\Gamma(t)} \in \QQ(t).
\end{align}
Since $\Gamma^{\sigma_{\gamma}}\subset \Gamma$ is a subgroup, it follows from (see \cite[Proposition 2.5, Theorem 7.7]{quillen}), that the rational function $P_{\Gamma, \gamma}(t)$ has no pole at $t = 1$, and thus the evaluation $P_{\Gamma, \gamma}(1) \in \QQ$ is well-defined. Hence, the proposition follows from Lemma \ref{lemma:powerseries} below.
\end{proof}
\begin{lemma} \label{lemma:powerseries}
 Let $f(t) = \sum_{i \geq 0} a_i t^i$ and $g(t) = \sum_{i \geq 0}g(t)$ be power series with $a_i, b_j \in \QQ_{\geq 0}$ and $a_0\neq 0$, $b_0 \neq 0$. For $N \geq 0$, define $C_N = (\sum_{i = 0}^Na_i)/(\sum_{i = 0}^N b_i) \in \QQ_{>0}$. 
 Assume that $f(t)$ and $g(t)$ are rational functions. Then the quotient $h(t) = f(t)/g(t)$ satisfies $h(1) = \lim_{N \to \infty} C_N$ provided that $h(t)$ has no pole at $t =1$.
\end{lemma}
\begin{proof}
    Since \( f(t), g(t) \in \mathbb{Q}(t) \), we can write
$
f(t) = \frac{P(t)}{(1 - t)^r Q(t)}$ and $g(t) = \frac{R(t)}{(1 - t)^s S(t)}$ where \( P(t), Q(t), R(t), S(t) \in \mathbb{Q}[t] \), and \( Q(1), S(1) \neq 0 \). The integers \( s \geq r \geq 0 \) are the pole orders at \( t = 1 \). %Hence:
%\[
%h(t) = \frac{f(t)}{g(t)} = \frac{P(t) S(t)}{Q(t) R(t)} (1 - t)^{s - r}
%\]
%We now study the asymptotics of the partial sums \( A_N \) and \( B_N \). 
From \cite[Theorem VI.1, p.~381]{flajolet}, we know that
$
a_n \sim \frac{P(1)}{Q(1)} \cdot \frac{n^{r - 1}}{(r-1)!}$ hence, by Faulhaber's formula,
$A_N \coloneqq \sum_{n=0}^N a_n \sim \frac{P(1)}{Q(1)} \cdot \frac{N^r}{r!}$. 
Likewise, $B_N \coloneqq \sum_{n=0}^N b_n \sim \frac{R(1)}{S(1)} \cdot \frac{N^s}{s!}$ so that
\[
\frac{A_N}{B_N} \sim \frac{P(1) S(1)}{Q(1) R(1)} \cdot \frac{(s-1)!}{(r-1)!}\cdot N^{r-s} \sim  \begin{cases}
  0 & \text{ if $s>r$} \\
  \frac{P(1) S(1)}{Q(1) R(1)}  & \text{ if $r=s$}.
\end{cases}
\]
In both cases, $A_N/B_N\sim h(1)$, proving the lemma.
\end{proof}

In fact, the proof of Proposition \ref{prop:yes-answer} shows something more precise. 
Let $\Gamma$ be a finite group and $\sigma \colon \Gamma \to \Gamma$ be an involution. For $\gamma \in \Gamma$ with $\gamma \sigma(\gamma)=e$, define $P_{\Gamma,\gamma}(t) \in \QQ(t)$ as in \eqref{align:rational-function-quillen}. Heuristically, the value $P_{\Gamma, \gamma}(1) \in \QQ$ reflects the ratio between the total mod $2$ Betti numbers of $B(\Gamma^{\sigma_\gamma})$ and $B\Gamma$. 

\begin{corollary} \label{cor:yes-answer}
Let $\Gamma$ be a finite $G$-group and let $\mr X = B\Gamma$. Then, we have:
    \begin{align*}
\varinjlim_{i \to \infty}   \left(\frac{\dim \rm H^{\leq i}_{\rm{grp}}(\mr X^G,\ZZ/2)}{\dim \rm H^{\leq i}_{\rm{grp}}(\mr X,\ZZ/2)}\right) = \sum_{[\gamma] \in \rm H^1(G,\Gamma)}\varinjlim_{i \to \infty}   \left(\frac{\dim \rm H^{\leq i}(\Gamma^{\sigma_\gamma},\ZZ/2)}{\dim \rm H^{\leq i}(\Gamma,\ZZ/2)}\right) = \sum_{[\gamma] \in \rm H^1(G,\Gamma)} P_{\Gamma,\gamma}(1).
    \end{align*}
\end{corollary}
This motivates the following weaker version of Question~\ref{que:groupoid-question}.\ref{q2}. 
%restricted to classifying groupoids with involution. 
%topological $G$-groupoids $\mr X$ of the form in the case $\mr X = B\Gamma$. 

\begin{question}\label{quest:quil-2}
Does there exist a constant $C > 0$ such that for every finite group $\Gamma$ and involution $\sigma \colon \Gamma \to \Gamma$, we have
$
\sum_{[\gamma] \in \rm H^1(G, \Gamma)} P_{\Gamma,\gamma}(1) \leq C$? 
\end{question}
A positive answer to Question~\ref{quest:quil-2} would provide a stepping stone toward a general positive answer to Question~\ref{que:groupoid-question}.\ref{q2}. However, even this zero-dimensional case is already quite subtle. For instance, the following example shows that one cannot take $C = 1$.

\begin{example}\label{prop:contrexampletoequivariantsmiththom}
Let $\mf S_4$ denote the symmetric group on four letters, equipped with the trivial $G$-action. The elements
$
\gamma_1 \coloneqq e, \gamma_2 \coloneqq (12)$ and $ \gamma_3 \coloneqq (12)(34)
$ form a complete set of representatives for the equivalence classes in $\rm H^1(G, \mf S_4)$. In particular, we have 
$
\# \rm H^1(G, \mf S_4) = 3.
$
Moreover, the fixed-point subgroups attached to $\gamma_1, \gamma_2$ and $\gamma_3$ are the subgroups
$\mf S_4^{\gamma_1} = \mf S_4$,  $\mf S_4^{\gamma_2} = \{e, (12), (34), (12)(34)\} \simeq \ZZ/2 \times \ZZ/2$, and $
\mf S_4^{\gamma_3} = \{e, (12), (34), (12)(34), (13)(24), (14)(23), (1423), (1324)\} \simeq D_8$, where $D_8$ denotes the dihedral group of order $8$. From \cite[Theorem 4.1]{MR154905} and \cite[Theorem 5.5]{MR1200878}, one computes the corresponding Poincar\'e series:
\[
P_{\ZZ/2 \times \ZZ/2}(t) = \frac{1}{(1 - t)^2}, \quad
P_{D_8}(t) = \frac{1}{(1 - t)^2}, \quad
P_{\mf S_4}(t) = \frac{1 + t^2}{(1 - t)^2(1 + t + t^2)}.
\]
Evaluating the associated ratios at $t = 1$, we obtain $
P_{\mf S_4, \gamma_1}(1) = 1$ and $P_{\mf S_4, \gamma_2}(1) = P_{\mf S_4, \gamma_3}(1) = 3/2$. Consequently, we have
$
\sum_{[\gamma] \in \rm H^1(G, \mf S_4)} P_{\mf S_4,\gamma}(1) = 1 + \tfrac{3}{2} + \tfrac{3}{2} = 4.
$
\end{example}

\subsection*{Declarations}
$$
$$
\textbf{Data availability statement}. No datasets were generated or analysed during the study on
which this paper is based.
\\
\\
\textbf{Conflict of interests statement}. On behalf of all authors, the corresponding author states
that there is no conflict of interest.

\begingroup
\sloppy
\printbibliography
\endgroup

\vspace{5mm}

\textsc{Emiliano Ambrosi, Institut de Recherche Mathématique Avancée (IRMA),
7 Rue René Descartes, 67084 Strasbourg, France}\par\nopagebreak 
  \textit{E-mail address:} \texttt{eambrosi@unistra.fr}

\vspace{5mm}

\textsc{Olivier de Gaay Fortman, Department of Mathematics,
       Utrecht University,
       Budapestlaan 6, 3584 CD
       Utrecht, The Netherlands}\par\nopagebreak 
  \textit{E-mail address:} \texttt{a.o.d.degaayfortman@uu.nl}

\end{document}